\documentclass[11pt]{article}
\usepackage[letterpaper,margin=1in]{geometry}
\usepackage[square,comma,sort&compress]{natbib}
\usepackage{amsmath,amsthm,mathtools}
\usepackage{fontspec}
\usepackage{unicode-math}
\AtBeginDocument{\renewcommand{\setminus}{\mathbin{\backslash}}}
\usepackage{graphicx,booktabs,tabularx,enumitem,xcolor,placeins}
\usepackage[bookmarksnumbered,bookmarksopen,unicode]{hyperref}
\usepackage{aliascnt}
\usepackage[capitalise,nameinlink,noabbrev]{cleveref}
\newcolumntype{Y}{>{\raggedright\arraybackslash}X}
\definecolor{paperdarkblue}{HTML}{17365D}
\hypersetup{colorlinks=true,allcolors=paperdarkblue}
\newcommand{\R}{\mathbb R}
\newcommand{\E}{\mathbb E}
\newcommand{\Prob}{\mathbb P}
\newcommand{\eps}{\varepsilon}
\newcommand{\Tr}{\mathord{\text{\normalfont\textsc{Tr}}}}
\newcommand{\diam}{\operatorname{diam}}

\newcommand{\spann}{\operatorname{span}}
\newcommand{\LMO}{\mathtt{LMO}}

\DeclareMathOperator*{\argmin}{arg\,min}

\theoremstyle{plain}
\newtheorem{theorem}{Theorem}[section]
\newcommand{\newsharedtheorem}[2]{%
  \newaliascnt{#1}{theorem}%
  \newtheorem{#1}[#1]{#2}%
  \aliascntresetthe{#1}%
}
\newsharedtheorem{lemma}{Lemma}
\newsharedtheorem{example}{Example}
\newsharedtheorem{proposition}{Proposition}
\newsharedtheorem{corollary}{Corollary}
\newsharedtheorem{conjecture}{Conjecture}
\theoremstyle{definition}
\newsharedtheorem{definition}{Definition}
\theoremstyle{remark}
\newsharedtheorem{remark}{Remark}
\crefname{theorem}{Theorem}{Theorems}
\Crefname{theorem}{Theorem}{Theorems}
\crefname{lemma}{Lemma}{Lemmas}
\Crefname{lemma}{Lemma}{Lemmas}
\crefname{example}{Example}{Examples}
\Crefname{example}{Example}{Examples}
\crefname{proposition}{Proposition}{Propositions}
\Crefname{proposition}{Proposition}{Propositions}
\crefname{corollary}{Corollary}{Corollaries}
\Crefname{corollary}{Corollary}{Corollaries}
\crefname{conjecture}{Conjecture}{Conjectures}
\Crefname{conjecture}{Conjecture}{Conjectures}
\crefname{definition}{Definition}{Definitions}
\Crefname{definition}{Definition}{Definitions}
\crefname{remark}{Remark}{Remarks}
\Crefname{remark}{Remark}{Remarks}
\crefname{section}{Section}{Sections}
\Crefname{section}{Section}{Sections}
\crefname{subsection}{Section}{Sections}
\Crefname{subsection}{Section}{Sections}
\crefname{subsubsection}{Section}{Sections}
\Crefname{subsubsection}{Section}{Sections}
\crefname{appendix}{Appendix}{Appendices}
\Crefname{appendix}{Appendix}{Appendices}
\AddToHook{cmd/appendix/after}{%
  \crefalias{section}{appendix}%
  \crefalias{subsection}{subappendix}%
  \crefalias{subsubsection}{subsubappendix}%
}
\crefname{subappendix}{Appendix}{Appendices}
\Crefname{subappendix}{Appendix}{Appendices}
\crefname{subsubappendix}{Appendix}{Appendices}
\Crefname{subsubappendix}{Appendix}{Appendices}
\newif\ificlredition
\iclreditionfalse

\newif\ifproofsinappendix
\proofsinappendixfalse
\ExplSyntaxOn
\seq_new:N \g_paper_proofs_seq
\cs_new_protected:Npn \paper_proof_inline:nn #1#2
  { \begin{proof}[Proof~of~\cref{#1}] #2 \end{proof} }
\cs_new_protected:Npn \paper_proof_defer:nn #1#2
  {
    \par\noindent\emph{Proof~of~\cref{#1}.}~See~\cref{proof:#1}.\par
    \seq_gput_right:Nn \g_paper_proofs_seq
      {
        \subsection{Proof~of~\texorpdfstring{\cref{#1}}{\ref*{#1}}}\label{proof:#1}
        \begin{proof} #2 \end{proof}
      }
  }
\NewDocumentEnvironment{paperproof}{O{auto} m +b}
  {
    \str_case:nnF {#1}
      {
        {inline}{\paper_proof_inline:nn {#2}{#3}}
        {appendix}{\paper_proof_defer:nn {#2}{#3}}
        {auto}{\ifproofsinappendix
          \paper_proof_defer:nn {#2}{#3}
          \else\paper_proof_inline:nn {#2}{#3}\fi}
      }
      {\PackageError{proof-placement}{Unknown~proof~placement:~#1}
        {Use~auto,~inline,~or~appendix.}}
  }{}
\NewDocumentCommand{\PrintDeferredProofs}{}
  {
    \seq_if_empty:NF \g_paper_proofs_seq
      {
        \section{Proofs~of~the~main~results}\label{app:deferred-proofs}
        \seq_map_inline:Nn \g_paper_proofs_seq {##1}
        \seq_gclear:N \g_paper_proofs_seq
      }
  }
\AtEndDocument{
  \seq_if_empty:NF \g_paper_proofs_seq
    {\PackageError{proof-placement}{Deferred~proofs~were~not~printed}
      {Call~PrintDeferredProofs~after~the~last~deferred~proof.}}
}
\ExplSyntaxOff
\proofsinappendixfalse
\hypersetup{pdfauthor={Sebastian Pokutta},
 pdftitle={Lower Bounds for Randomized Linear-Oracle Optimization\\ over Strongly Convex Sets}}
\title{Lower Bounds for Randomized Linear-Oracle Optimization\\ over Strongly Convex Sets}
\author{Sebastian Pokutta\\
 Zuse Institute Berlin and TU Berlin\\
 \texttt{pokutta@zib.de}}
\date{\today}
\begin{document}
\maketitle
\begin{abstract}
Recent results by \citet{halbey2026lower} and \citet{grimmer2026lower} establish
lower bounds for optimization over strongly convex sets, covering specified
Frank--Wolfe updates and unrestricted deterministic first-order linear
minimization oracle (LMO) methods, respectively. Neither establishes a lower
bound for unrestricted randomized methods.
We close this gap by proving an $\Omega(L\diam(S)^2/T^2)$ distributional lower bound after $T$ oracle calls for randomized first-order LMO methods over strongly convex sets, with final outputs required to be feasible on every instance and seed. While our construction reuses the family of \citet{grimmer2026lower}, our proof characterizes the conditional distribution of the hidden permutation after adaptive queries and bounds the expected decrease of a combinatorial statistic controlling optimizer variance. We also show that the same error threshold holds with probability at least $3/4$ and yields a stopping-time tail bound.
Our construction uses the same curvature--diameter coupling as
\citet{grimmer2026lower}. Comparisons with \citet{garber2015faster} must
account for this coupling. Separately, we show that prescribed slow exact-line-search
Frank--Wolfe trajectories can locally amplify sufficiently small initialization
errors, corresponding to roughly one extra bit of initialization accuracy per
iteration, and that initializations satisfying a particular eventual slow envelope
have planar measure zero.
\end{abstract}

\section{Introduction}
\label{sec:introduction}
We are interested in solving the optimization problem
$$
\min_{x \in S} f(x),
$$
where $f$ is an $L$-smooth convex function and $S$ is a compact convex set. The Frank--Wolfe algorithm \citep{frank1956algorithm,levitin1966constrained} replaces potentially expensive projections by a linear minimization oracle (LMO), $\LMO_S(p)\in\argmin_{z\in S}\langle p,z\rangle$ and its
general $O(1/T)$ rate is optimal in the linear-oracle model
\cite{jaggi2013revisiting,lan2013complexity}. However, once we add more structure to both the function and the domain, convergence rates can often be improved; see \citep{CGFWSurvey2022} for an overview. In particular, in the case of strongly convex functions and strongly convex sets, a rate of $O(1/T^2)$ was established in \citet{garber2015faster} and similarly improved convergence rates can be obtained for sharp functions \citep{KDP2018jour}, uniformly convex sets \citep{KAP2020,KAP2021}, and even if the objective is just convex but the set satisfies a mild curvature condition, $o(1/T)$ rates can be obtained \citep{P2026FrankWolfe}. We will be concerned with \emph{distributional} lower bounds for \emph{randomized} LMO-based first-order methods for the case of strongly convex functions over strongly convex sets.

\subsection*{Related work}
\label{sec:existing-bounds}
Most closely related to our work are \citet{halbey2026lower} (and their extension to uniformly convex sets and sharp functions \citep{HRP2026CurvatureFWLowerBounds}) as well as \citet{grimmer2026lower}; we summarize the differences in \cref{tab:lower-bound-scope}.

\citet{halbey2026lower} give slowly converging trajectories at a rate of $\Omega(1/t^2)$ for exact-line-search or short-step
Frank--Wolfe, with a horizon-dependent prescribed initialization. Their result concerns those specific updates, not arbitrary LMO methods. Moreover, as we show in \cref{sec:halbey-diagnosis}, tiny changes in the starting point can be strongly amplified along the slow exact-line-search trajectories. Over a fixed horizon of $T$ iterations, sufficiently small initialization errors can be magnified by roughly $2^T$, up to polynomial factors. Offsetting this local amplification corresponds to roughly one additional binary digit of initialization accuracy per iteration. This describes sensitivity near a prescribed trajectory; it does not mean that every perturbation or rounding error makes Frank--Wolfe converge faster. We also establish the existence of \emph{infinite} slow-envelope trajectories on their unit-disk instance and show that the initializations with this eventual behavior form an uncountable set of planar Lebesgue measure zero (\cref{thm:halbey-measure-zero}); the stated guarantee in \citet{halbey2026lower} is finite-horizon, with an initialization that may depend on the horizon.

Complementing the above, for any \emph{deterministic} LMO-based method returning feasible outputs, \citet{grimmer2026lower} establish an $\Omega(L\diam(S)^2/(T+1)^2)$ objective-error lower bound after $T$ oracle calls. Their resisting oracle completes a hard instance \emph{after} observing the algorithm's queries. Applying this argument separately to each random seed may produce a different hard instance for each seed. A randomized lower bound instead requires an instance fixed independently of the realized seed, on which the algorithm has large expected error. On the other hand \citet{grimmer2026lower} also give one fixed instance hard for every LMO-span method, which is a subclass of all LMO-based methods. This result immediately gives rise to a lower bound for randomized methods but \emph{only} those from the restricted class of LMO-span methods, which is a significant restriction as we will discuss in \cref{sec:grimmer-liu-gap}.

\begin{table}[t]
\caption{Comparison of the lower bounds. The rate is
objective error after $T$ oracle calls.}
\label{tab:lower-bound-scope}
\centering
\small
\setlength{\tabcolsep}{4pt}
\renewcommand{\arraystretch}{1.15}
\begin{tabularx}{\textwidth}{@{}>{\raggedright\arraybackslash}p{2.3cm}YYY>{\centering\arraybackslash}p{1cm}@{}}
\toprule
Result & Method class & Randomization covered? & Data chosen & Rate \\
\midrule
\citep{halbey2026lower} & Exact-line-search or short-step Frank--Wolfe
  & Not an unrestricted-method result
  & Horizon-dependent initialization
  & $T^{-2}$ \\
\citep{grimmer2026lower}, span & LMO-span methods
  & Yes, for every fixed seed
  & One fixed ordered instance
  & $T^{-2}$ \\
\citep{grimmer2026lower}, resisting & Unrestricted first-order LMO
  & Deterministic only
  & Permutation completed from the transcript
  & $T^{-2}$ \\
\textbf{This paper} & Unrestricted first-order LMO with feasible final output on every instance and seed
  & Yes
  & Uniform prior fixed before the seed
  & $T^{-2}$ \\
\bottomrule
\end{tabularx}
\end{table}
\FloatBarrier

\subsection*{Contribution} We extend the lower bound
of \citet{grimmer2026lower} from deterministic to randomized methods,
preserving the diameter-normalized order of the lower bound and the geometric parameter regime (up to numerical constants). The hard distribution we construct is fixed ahead of time and the algorithm can have full knowledge of it. In particular, it is independent of the algorithm's seed. We do not require any restrictions on either the query directions or the intermediate evaluation points. Yao's principle \citep{yao1977probabilistic} then gives:
\begin{equation}\label{eq:intro-main-rate}
 \E[f(\widehat x_T)-f^\star]
 =\Omega\!\left(\frac{L\diam(S)^2}{(T+1)^2}\right).
\end{equation}
We follow the general information-theoretic approach from \citet{braun2015information,BPX2014jour}, where we bound the (expected) information gain in each iteration albeit we do not work with information-theoretic quantities such as mutual information and entropy but rather use a combinatorial statistic $\Delta$ through which we lower-bound the conditional optimizer variance.

We retain the feasible sets, common quadratic objective, optimizer and thresholding formula of \citet{grimmer2026lower}, and their conversion of optimizer variance into loss.
The departure is the \emph{fixed uniform prior} over their feasible sets, which turns out to be our hard distribution. In particular, in contrast to the resisting oracle setup, we can no longer respond adversarially to the queries of the algorithm. As a consequence an answer can now potentially reveal significantly more information than in the resisting oracle setup.

We use a uniform permutation prior because it allows us to obtain an exact
description of the information revealed by LMO answers and of the resulting posterior
after adaptive queries. This then provides the bound on the expected progress, which is key to our proof. Using random embedding could be a potential alternative route, but would require its own control of information revealed by unrestricted exact queries; the span lower bound of \citet{grimmer2026lower} alone does not provide that and exploration of such is beyond the scope of this paper.

The hidden instance is determined by a permutation of a known list of $d$
coordinate weights. Let $\Tr_t$ record the algorithm's seed and its first $t$
queries and answers, and let $m_t\le d$ count the coordinate--weight
assignments not yet explicitly disclosed by a nonzero LMO component.
The permutations compatible with this transcript form a graph under
exchanges of two coordinate weights. We show that this graph has a common
outdegree $\Delta_t$, the number of compatible exchanges from any one of its
vertices. Our key estimate bounds its expected decrease in one query:
\begin{equation}\label{eq:intro-main-estimation}
   \E[\Delta_t-\Delta_{t+1}\mid\Tr_t]\le6m_t.
\end{equation}
Initially $m_0=d$ and $\Delta_0=\binom d2$. Since $m_t\le d$, summing the
expected decreases gives
\[
 \E\Delta_T\ge\binom d2-6dT.
\]
Thus the initial potential is quadratic in $d$, whereas the expected decrease
per query is at most linear in $d$. Choosing $d$ to be a sufficiently large
constant multiple of $T+1$ leaves substantial conditional optimizer variance,
which we convert into optimization error.
A single query may disclose several coordinate weights; our argument bounds
expected progress rather than imposing a one-coordinate-per-query rule.
We also do not require a span restriction.
\Cref{sec:model-family-result} collects the results and their short derivation;
\cref{sec:exact-transcripts,sec:posterior-error} identify the conditional distribution
and its variance; \cref{sec:one-query-progress} gives the full $6m$ argument.

\section{Model, family, and lower bounds}
\label{sec:model-family-result}
\subsection{Oracle model}
For a compact convex set $S\subset\R^d$ and differentiable
$f:\R^d\to\R$, we consider
\begin{equation}\label{eq:problem}
 f^\star\doteq\min_{x\in S}f(x).
\end{equation}
The function is $L$-smooth when its gradient is $L$-Lipschitz, and
$\mu$-strongly convex when
$f(y)\ge f(x)+\langle\nabla f(x),y-x\rangle+\mu\|y-x\|_2^2/2$
for all $x,y$.
\begin{definition}[Strongly convex set]\label{def:strong-set}
A convex set $S$ is $\alpha$-strongly convex if, for $x,y\in S$ and
$\theta\in[0,1]$,
\[
 \theta x+(1-\theta)y+
 B\!\left(0,\frac\alpha2\theta(1-\theta)\|x-y\|_2^2\right)\subseteq S,
\]
where $B(0,r)$ is the closed Euclidean ball.
\end{definition}
A \emph{linear minimization oracle (LMO)} for $S$ returns, for each direction
$p\in\R^d$, a point
\begin{equation*}
 \LMO_S(p)\in\argmin_{z\in S}\langle p,z\rangle.
\end{equation*}
A \emph{query} consists of any $(x_t,p_t)\in\R^d\times\R^d$ and returns
$y_t\doteq(f(x_t),\nabla f(x_t),z_{t+1})$, where $z_{t+1}\doteq\LMO_S(p_t)$.
Let $U$ be any \emph{random seed} independent of the hidden instance, and put
\begin{equation}\label{eq:transcript}
 \Tr_t\doteq(U,(x_s,p_s,y_s)_{s<t}),\qquad
 \mathcal H_t\doteq\sigma(\Tr_t).
\end{equation}
Here $\Tr_t$ is the \emph{transcript}, including the seed, and $\mathcal H_t$ is
its generated $\sigma$-algebra: all information available before query $t$;
conditioning on $\Tr_t$ means conditioning on $\mathcal H_t$.
Fixing a realization $U=u$ makes the algorithm's decision rules deterministic;
the transcript may still vary with the hidden instance.
Queries, stopping decisions and outputs are arbitrary measurable functions
of this history. Our argument is purely information-theoretic with minimal assumptions on the randomized algorithm in question beyond first-order and LMO access. In particular, we do not assume any span, atom-hull, computational, or finite-seed restriction. Evaluation points $x_t$ need not belong to $S$: the function and gradient are defined on all of $\R^d$. Only the returned output must be
feasible, on \emph{every instance and every seed}, not merely almost surely.
Ordinary Frank--Wolfe and its feasible variants meet this requirement.
The model provides only the oracle responses specified above; an
exact-line-search oracle for an unknown objective would be additional access.
In our hard family, however, the objective is the same known quadratic for
every permutation, with all its parameters independent of the hidden instance. Thus, evaluating this formula, or computing an exact line search from it, requires no additional oracle access. Only the feasible set is hidden in this family. 
We set $\LMO_{S_\pi}(0)\doteq0$; unused calls can be padded with zero queries.

\subsection{The permuted hard family}

The following construction, including its normalization, is due to
\citet[Sections 2.1 and 3.1]{grimmer2026lower}. Fix $d\ge3$ and define
\begin{equation}\label{eq:constants}
  C \doteq \frac1{\sqrt{d^2+d+2}},
  \qquad w_k \doteq C\sqrt{2k},\quad k=1,\ldots,d.
\end{equation}
For a permutation $\pi$ of $\{1,\ldots,d\}$, let
\begin{equation}\label{eq:set-family}
  S_\pi \doteq \left\{x\in\R^d:
  \frac12\|x\|_2^2+\sum_{i=1}^d w_{\pi(i)}|x_i|\le C^2\right\}.
\end{equation}
Coordinate $i$ carries weight $w_{\pi(i)}$; the integer $\pi(i)$ is its
\emph{weight rank}. We have $2C^2+\sum_k w_k^2=1$.

\begin{lemma}[Geometry of $S_\pi$]\label{lem:geometry}
Every $S_\pi$ is $1$-strongly convex and
$\diam(S_\pi)=2C(2-\sqrt2)$.
\end{lemma}

See the strong-convexity and diameter lemmas of
\citet[Section 2.1]{grimmer2026lower}.
Next we pick $\rho\in(0,1)$ to be the unique solution of the boundary equation
\begin{equation}\label{eq:rho}
 \sum_{k=1}^d\left[
  \frac12\rho^2(w_d-w_k)^2+\rho w_k(w_d-w_k)
 \right]=C^2,
\end{equation}
whose existence and uniqueness are proved in \citet[Section 3.1]{grimmer2026lower}.
Define
\begin{equation}\label{eq:objective}
 M \doteq \frac{\rho}{1-\rho}w_d,
 \qquad f(x) \doteq \frac12\|x-M\mathbf1\|_2^2.
\end{equation}
This objective is $1$-smooth, $1$-strongly convex, and used for all instances.
Moreover, for every instance $\pi$, the choice of $\rho$ in \eqref{eq:rho}
places the constrained minimizer on the boundary of $S_\pi$.

\begin{lemma}[Optimal points]\label{lem:optima}
The minimizer of $f$ over $S_\pi$ is unique, with
\begin{align}
  (x_\star^{(\pi)})_i =\rho\bigl(w_d-w_{\pi(i)}\bigr) \qquad \text{and} \qquad
  \rho&\ge\frac2{d^2}.\label{eq:optimum} %
\end{align}
\end{lemma}

See the permuted-minimizer lemma of \citet[Section 3.1]{grimmer2026lower}.

\begin{remark}[Common optimal value]\label{rem:common-optimum-value}
Coordinate permutations preserve the common quadratic objective. Although
the optimal points vary with $\pi$, all instances have the same optimal value
\[
 f^\star=\frac12\sum_{k=1}^d
       \bigl(\rho(w_d-w_k)-M\bigr)^2.
\]
For a fixed terminal transcript, including the seed, the output $x$ is fixed:
the output rule is a deterministic function of that transcript. Its loss
$f(x)-f^\star$ is consequently constant over all compatible instances, so its
conditional expected loss equals its actual loss on each completion.
\end{remark}

For admissible deterministic methods $A$ with budget $T$, measurable
mixtures $\nu$, and objective-gap loss $\mathcal L$, put
\[
 \begin{aligned}
 V_{\rm dist}(\mu) \doteq\inf_A\E_{I\sim\mu}\mathcal L(A,I) \qquad \text{and} \qquad
 V_{\rm rand} \doteq\inf_\nu\sup_I\E_{A\sim\nu}\mathcal L(A,I).
 \end{aligned}
\]
The method $A$ may know $\mu$. Yao's inequality $V_{\rm dist}(\mu)\le V_{\rm rand}$ makes a
fixed-prior bound sufficient. \Cref{sec:minimax} gives the full
minimax formulation and proof.
\begin{proposition}[Uniform prior]\label{prop:uniform-optimal}
For the permuted family,
\begin{equation}\label{eq:uniform-optimal}
  V_{\rm dist}(\operatorname{Unif})=V_{\rm rand}.
\end{equation}
\end{proposition}
\begin{paperproof}{prop:uniform-optimal}
For a coordinate permutation $g$, let $P_g$ be its permutation matrix.
Relabeling sends $S_\pi$ to
$P_gS_\pi=S_{\pi\circ g^{-1}}$, and $P_g\mathbf1=\mathbf1$ makes the
objective invariant. To run the conjugated method $gA$ on $S$, simulate $A$
on $P_g^{-1}S$: transform internal queries $(x,p)$ to $(P_gx,P_gp)$ and
transform observed gradients and LMO answers back by $P_g^{-1}$. Function
values stay unchanged, and the final internal output is mapped by $P_g$.
This preserves feasibility, objective gap, and query count. Thus the method
class and loss are invariant under a transitive action on the $d!$ instances.
The inequality $V_{\rm dist}(\operatorname{Unif})\le V_{\rm rand}$ follows
from \eqref{eq:yao-inequality}. Conversely, fix a deterministic method $A$,
choose $g$ uniformly, and run $gA$. On any fixed instance, invariance and
transitivity make its expected loss equal to the uniform-prior average loss
of $A$. This randomized method therefore has the same expected loss on every
instance. Taking the infimum over $A$ gives
$V_{\rm rand}\le V_{\rm dist}(\operatorname{Unif})$.
\end{paperproof}

\begin{lemma}[Rescaling]\label{lem:rescaling-paper}
For $L,\alpha>0$, the rescaling $\widetilde S\doteq\alpha^{-1}S$ and $\widetilde f(x)\doteq\frac{L}{\alpha^2}f(\alpha x)$
takes a $1$-smooth, $1$-strongly convex $f$ and $1$-strongly convex $S$
to an $L$-smooth, $L$-strongly convex objective and an $\alpha$-strongly convex
set. Diameter $D$ and gap $g$ become $D/\alpha$ and $Lg/\alpha^2$;
LMO transcripts correspond under $x\mapsto\alpha x$.
\end{lemma}

See the rescaling lemma of \cite[Section 2]{grimmer2026lower}. For completeness we include a proof due to the additional measurable maps that need to be preserved.

\begin{paperproof}{lem:rescaling-paper}
Fix a seed of a method for the scaled problem and simulate it using the
base oracle. Send its query $(x,p)$ to $(\alpha x,p)$; feed back the base
value and gradient multiplied by $L/\alpha^2$ and $L/\alpha$, and the
base LMO answer divided by $\alpha$. If the scaled method returns
$\widehat x$, the base simulator returns $\alpha\widehat x$.
Its base gap multiplied by $L/\alpha^2$ is the scaled gap.
These measurable maps preserve the first stopping decision and universal
final feasibility, without restricting the method's directions.
\end{paperproof}

\subsection{Derivation from the two main estimates}\label{sec:consequences}
In the later sections we will construct a combinatorial statistic $\Delta_t \geq 0$ from a given transcript, with $\Delta_0=d(d-1)/2$ and $\Delta_{t+1}\le\Delta_t$. For this quantity we will establish two key estimates
\begin{equation}\label{eq:proof-roadmap}
 \E[\Delta_t-\Delta_{t+1}\mid\mathcal H_t] \leq 6d
 \qquad \text{and} \qquad
 f(\widehat x_T)-f^\star \geq \frac{\rho^2 C^2}{48d^2} \Delta_T^2.
\end{equation}
The first is \cref{cor:one-query}, using $m_t\le d$. It bounds the expected
decrease of our combinatorial uncertainty statistic, averaging over hidden
permutations compatible with the current transcript; the seed and past
observations are already fixed.
The second follows from \cref{cor:conditional-error} and the common
optimal value in \cref{rem:common-optimum-value}. It converts the lower bound
on conditional optimizer variance measured by $\Delta_T$ into a lower bound on the primal gap; it is trivial when
no coordinates remain unresolved. Our main work will be to bound the
expected decrease.

\subsection{Main results}
In the considered model here the method knows the common objective, the scalar parameters and the uniform prior. Only the permutation, sampled independently of the seed, is hidden.

\begin{theorem}[Distributional lower bound]\label{thm:main-paper}
For every integer $T\ge1$, let $d \doteq 64(T+1)$ and let $\Pi$ be uniform
on the $d!$ sets \eqref{eq:set-family}. Every randomized first-order LMO
method using at most $T$ calls and returning a feasible point on every
instance and every seed satisfies
\begin{equation}\label{eq:main-paper}
 \E_{\Pi,U}[f(\widehat x_T)-f(x_\star^{(\Pi)})]
 \ge c\,\frac{\diam(S_\Pi)^2}{(T+1)^2},
 \qquad c \doteq (3+2\sqrt2)/(6144\cdot64^2).
\end{equation}
For every $L,\alpha>0$, positive rescaling gives the right side
$cL\diam(S)^2/(T+1)^2$, an $L$-smooth, $L$-strongly convex objective and
an $\alpha$-strongly convex feasible set. Some fixed permutation has at
least this expected loss over the seed alone.
\end{theorem}

\begin{paperproof}{thm:main-paper}
Fix the seed. Summing the first estimate gives
\begin{equation}\label{eq:remaining-deficit}
 \E\Delta_T\ge\Delta_0-6Td\ge d^2/8
 \qquad(d=64(T+1)),
\end{equation}
since $\Delta_0\ge d^2/4$ and $6Td\le d^2/8$.
The second estimate and Jensen's inequality then yield
\[
 \E[f(\widehat x_T)-f^\star]
 \ge\frac{\rho^2C^2d^2}{3072}
 \ge\frac{C^2}{768d^2}
 =\frac{3+2\sqrt2}{6144}\frac{\diam(S_\pi)^2}{d^2}.
\]
Using $\rho\ge2/d^2$ and
$C^2=(3+2\sqrt2)\diam(S_\pi)^2/8$.
Substitute $d=64(T+1)$, average over the independent seed, and then select
a fixed instance from the finite average. \Cref{lem:rescaling-paper}
gives every $L,\alpha>0$.
\end{paperproof}

\begin{corollary}[Query complexity]\label{cor:complexity-paper}
There are instances requiring
$T=\Omega((L\diam(S)^2/\eps)^{1/2})$ queries to guarantee expected gap at
most $\eps$, under the same universal final-feasibility requirement.
\end{corollary}
\begin{proof}
Rearrange \eqref{eq:main-paper}:
$(T+1)^2\ge cL\diam(S)^2/\eps$ is necessary.
\end{proof}
Here $\diam(S)$ is that of the constructed family in
\cref{thm:main-paper}.

\begin{theorem}[Constant-probability lower bound]
\label{thm:constant-probability}
Put $c_0 \doteq \frac{3+2\sqrt2}{6144\cdot64^2}$. Under the dimension, prior, rescaling and method
assumptions of \cref{thm:main-paper},
\begin{equation}\label{eq:constant-probability}
 \Prob_{\Pi,U}\!\left[f(\widehat x_T)-f(x_\star^{(\Pi)})
 \ge c_0\frac{L\diam(S)^2}{(T+1)^2}\right]\ge \frac{3}{4}.
\end{equation}
Some fixed permutation has probability at least $\frac{3}{4}$ over the seed alone.
\end{theorem}

\begin{paperproof}{thm:constant-probability}
Let $D_T \doteq \Delta_0-\Delta_T\ge0$. \Cref{eq:proof-roadmap} gives
$\E D_T\le6Td$, also after seed averaging. For
$a\doteq\Delta_0-d^2/8=d(3d-4)/8>0$, Markov's inequality gives
\begin{equation}\label{eq:large-deficit-probability}
 \Prob[\Delta_T<d^2/8]=\Prob[D_T>a]
 \le\frac{6Td}{a}=\frac{48T}{3d-4}\le\frac14,
\end{equation}
because $3d-4=192T+188$. On the complementary event,
\eqref{eq:proof-roadmap} gives gap at least
$\rho^2C^2d^2/3072\ge c_0\diam(S_\pi)^2/(T+1)^2$.
The loss is constant on each seed/transcript class, so this is an actual
loss bound, not only a conditional expectation. Positive rescaling and
finite averaging give the stated probability and a fixed-instance witness.
\end{paperproof}

\begin{corollary}[Stopping-time tail]\label{cor:stopping-tail}
Let $N$ count calls up to the first return, with $N\doteq\infty$ on
nontermination. Every returned output must be feasible on every instance
and every seed. If, on each instance, the probability of returning an
output whose gap is at least the threshold in
\cref{thm:constant-probability} is at most $\eta< \frac{3}{4}$, then some
fixed instance satisfies $\Prob[N>T]\ge \frac{3}{4} -\eta$.
\end{corollary}
\begin{proof}
Truncate at $T$, returning the feasible point $0$ if the method has not
stopped. The truncated failure event is contained in the union of the
original returned-failure event and $\{N>T\}$. Apply
\cref{thm:constant-probability}; nontermination already lies in the tail.
\end{proof}

The fixed permutation witnesses for the expectation in \cref{thm:main-paper} and the probability in \cref{thm:constant-probability} may differ; neither
depends on the realized seed though. The diameter in our bound is that of the constructed instance.
Recall that we use the same family as in \citet{grimmer2026lower}. This family satisfies $\alpha\diam(S)=\Theta(1/d)$, so curvature and diameter cannot be rescaled
independently. With $d=64(T+1)$ and fixed $L,\alpha$, the value of our lower bound is of
order $L/(\alpha^2(T+1)^4)$. Our randomized extension retains the same geometric
regime: it matches the diameter-normalized order of the lower bound in \citet{grimmer2026lower}. The general upper bound of \citet{garber2015faster} is not tight on this
family: returning $0$ without queries already achieves an error of the same
order as our lower bound. For short-step or exact-line-search Frank--Wolfe started at $0$,
the error cannot increase, so the bound obtained at $0$ also holds after
any number of iterations. Our lower bound shows that $T$ queries cannot
guarantee an improvement by more than a constant factor in worst-case
expected error. We provide a detailed discussion in \Cref{app:parameter-comparison}.

\section{Exact LMO transcripts under a uniform permutation}
\label{sec:exact-transcripts}

For the considered family of instances, the exact thresholding formula captures all information revealed in an LMO answer, including the inequalities implied by its zero coordinates. With this we can describe the entire conditional distribution directly.

\subsection{The exact oracle answer}

Recall the convention $\LMO_{S_\pi}(0)=0$. A zero query reveals no
instance-dependent information.

We recall the exact LMO formula of
\cite[Sections~2.2 and~3.1]{grimmer2026lower} and derive an explicit
characterization of the information contained in its answer.

\begin{lemma}[Exact LMO and the information in its answer]
\label{lem:exact-lmo}
Fix $p\ne0$ and let $z\doteq\LMO_{S_\pi}(p)$. There is a unique $\lambda>0$ such
that
\begin{equation}\label{eq:lmo}
 z_i=-\operatorname{sign}(p_i)
 \max\left\{0,\frac{|p_i|}{\lambda}-w_{\pi(i)}\right\}.
\end{equation}
Define the \emph{threshold parameter} $\tau$ and \emph{active set} $A$ by
\begin{equation}\label{eq:tau}
  \tau \doteq \frac1{2C^2\lambda^2},
  \qquad A \doteq \{i:\tau p_i^2>\pi(i)\}.
\end{equation}
Coordinates in $A$ are \emph{active}; the others are \emph{inactive}. Then
$A=\{i:z_i\ne0\}$ and
\begin{equation}\label{eq:water-level}
  \sum_{i=1}^d(\tau p_i^2-\pi(i))_+=1.
\end{equation}
The pair $(p,z)$ determines $\lambda$, $\tau$, every active rank $\pi(i)$,
and the inactive lower bounds
\begin{equation}\label{eq:lower-bound-answer}
  \pi(i)\ge\lceil\tau p_i^2\rceil
  \qquad(i\notin A).
\end{equation}
Conversely, the active-rank equalities and inactive-rank lower bounds determine
the LMO answer.
\end{lemma}
We can recover $\lambda$ and $\pi$ explicitly from the observed vectors:
\begin{equation}\label{eq:observed-decoder}
 \lambda=\frac{\sum_i|p_i z_i|}{C^2+\|z\|_2^2/2},\qquad
 \pi(i)=\frac{(|p_i|/\lambda-|z_i|)^2}{2C^2}\quad(z_i\ne0).
\end{equation}
The coordinate label $i$ is already known; its assigned weight rank is
learned. Every nonzero component supplies an equality, not just one per
query; some equalities may repeat previously revealed ranks.
\begin{paperproof}{lem:exact-lmo}
The thresholding formula and uniqueness of its multiplier are
the exact-LMO lemma from \citet[Section 2.2]{grimmer2026lower},
applied after a coordinate permutation. We prove the information statements.
For $z_i\ne0$, put $u_i\doteq|p_i|/\lambda$. The coordinate's constraint
contribution is
\[
 \frac12(u_i-w_{\pi(i)})^2
 +w_{\pi(i)}(u_i-w_{\pi(i)})
 =\frac12\bigl(u_i^2-w_{\pi(i)}^2\bigr).
\]
Using $w_k^2=2C^2k$ and dividing the boundary equation by $C^2$ gives
\eqref{eq:water-level}. Activity is equivalent to
$\tau p_i^2>\pi(i)$. The left side of \eqref{eq:water-level} is continuous,
strictly increasing once positive, and ranges from zero to infinity, so it
also determines $\tau$ uniquely.

On active coordinates, \eqref{eq:lmo} gives
$w_{\pi(i)}=|p_i|/\lambda-|z_i|$. Substitution into the constraint yields
\[
 C^2
 =\frac1\lambda\sum_i|p_i z_i|-\frac12\|z\|^2,
 \qquad
 \lambda=\frac{\sum_i|p_i z_i|}{C^2+\tfrac12\|z\|^2}.
\]
In particular, the observed vectors alone give
\[
 \pi(i)=\frac{(|p_i|/\lambda-|z_i|)^2}{2C^2}\quad(z_i\ne0).
\]
Here $i$ is an already known coordinate label; the learned quantity is the
weight rank assigned to it which induces one equality condition for each nonzero component. Some may repeat earlier revealed ranks. For an inactive coordinate, stationarity gives
$|p_i|\le\lambda w_{\pi(i)}$, equivalently
$\tau p_i^2\le\pi(i)$ and \eqref{eq:lower-bound-answer}.

Conversely, with the active set specified, \eqref{eq:water-level} gives
\begin{equation}\label{eq:tau-from-active}
  \tau=\frac{1+\sum_{i\in A}\pi(i)}{\sum_{i\in A}p_i^2}.
\end{equation}
The active ranks and inactive inequalities reproduce this $\tau$, the same
activity pattern, and the vector in \eqref{eq:lmo}.
\end{paperproof}

\subsection{Conditional distributions, Remaining ranks, and compatible assignments}

Let $\Pi$ be uniform over all $d!$ permutations, selected before the method
starts. Queries may depend on the complete previous transcript.

\begin{lemma}[Form of every conditional posterior]\label{lem:posterior-form}
Fix a seed value $U=u$, thereby fixing a deterministic method, and a
transcript $\Tr_t$ of positive probability under the uniform instance prior
for that fixed method. Conditional on this transcript, $\Pi$ is uniform over
all permutations satisfying:
\begin{enumerate}[label=(\roman*),leftmargin=2em]
  \item exact assignments $\Pi(i)=k$ for coordinates active in an earlier
        LMO answer;
  \item lower bounds $\Pi(i)\ge b_i$ from coordinates inactive in earlier
        answers.
\end{enumerate}
There are no additional restrictions in the numerical LMO answers.
\end{lemma}

\begin{paperproof}{lem:posterior-form}
At time zero there are no constraints. Suppose the statement holds before a
query. For fixed $U$ and $\Tr_t$, the pair $(x_t,p_t)$ is fixed and adds no
condition on $\Pi$. The value and gradient of the common objective
\eqref{eq:objective} are independent of $\Pi$. By
\cref{lem:exact-lmo}, the LMO answer is equivalent exactly to the active
rank equalities and inactive lower bounds. Intersecting them with the previous
constraints proves the description by induction. For the fixed seed, the
transcript is a deterministic function of $\Pi$. Conditioning on it therefore
restricts the uniform prior to the compatible permutations, preserving their
equal weights.
\end{paperproof}

Fix the seed and a realized transcript. Let $K_t$ be the coordinates whose
ranks appeared in nonzero LMO components. From the query and its answer,
every nonzero component $z_i$ reveals $w_{\Pi(i)}$, and hence $\Pi(i)$.
Let $I_t \doteq [d]\setminus K_t$.
Write $m \doteq |I_t|$ and list the unused \emph{original weight ranks} as
$r_1<\cdots<r_m$. Here $j$ denotes a position in the remaining list, while
$r_j$ denotes the corresponding original rank. The remaining weights are
therefore $w_{r_1},\ldots,w_{r_m}$. This lets us work with a shorter index set
without changing the original ranks or weights.
For $v\in I_t$, let $b_v$ be its largest observed lower
bound (initially $1$). The first remaining rank at least $b_v$ has index $L_v\doteq\min\{j\in[m]:r_j\ge b_v\}$. Realizability ensures the minimum exists. Let $\iota:[m]\to I_t$ list coordinates in increasing order of $(L_v,v)$, so ties retain the original
coordinate order, and put $\ell_i\doteq L_{\iota(i)}$.
A \emph{reduced permutation} $\sigma$ means precisely $\Pi(\iota(i))=r_{\sigma(i)}$.
Thus representing the permutations as permutation matrices, $v\doteq\iota(i)$ is a \emph{coordinate} (not to be confused with a list index $i$ here), $i$ its row number,
$j\doteq\sigma(i)$ a column number, and $r_j$ the original integer weight rank.
Query components below are relabeled by the same map:
$p_i$ in a reduced formula means the full query component $p_{\iota(i)}$. See \cref{app:assignment-example} for an example.

Conditional on the full transcript, including the seed, the induced reduced
permutation $\Sigma$ is uniformly distributed over the compatible assignments
\begin{equation}\label{eq:nested-family}
 \mathcal F(\ell)\doteq\{\sigma\in\mathfrak S_m:
                  \sigma(i)\ge\ell_i\ \text{for all }i\}.
\end{equation}
A permutation is represented by one occupied cell in each row and column.
Row $i$ permits columns $\ell_i,\ldots,m$; these allowed sets are nested as the $\ell_i$ are increasing and form a Ferrers board. With $\ell_1\le\cdots\le\ell_m$, nonemptiness is equivalent to $\ell_i\le i$: if $\ell_i>i$, rows $i,\ldots,m$ have too few columns; otherwise assign from row $m$ down to row $1$. Set $c_i\doteq i-\ell_i+1$ then we have

\begin{lemma}[Counting compatible assignments]\label{lem:count}
For $\mathcal F(\ell) \neq \emptyset$ we have
\begin{equation}\label{eq:count}
 |\mathcal F(\ell)|=\prod_{i=1}^m c_i.
\end{equation}
\end{lemma}
\begin{paperproof}{lem:count}
Fill rows $m,m-1,\ldots,1$. The $m-i$ already assigned rows have lower
bounds at least $\ell_i$, so they remove exactly $m-i$ columns from
$\ell_i,\ldots,m$. Row $i$ therefore has $c_i$ available columns, whatever
the preceding choices. Multiplication proves \eqref{eq:count}. Choosing
uniformly among the currently available allowed columns gives every
compatible assignment probability $\prod_i c_i^{-1}=|\mathcal F(\ell)|^{-1}$.
\end{paperproof}

\section{Conditional variance and optimization error}
\label{sec:posterior-error}
The deterministic proof in \cite[Section 3.3]{grimmer2026lower} averages over an unconstrained suffix of unassigned ranks and obtains a lower bound on loss via their optimizer variance. We use the same variance-to-loss argument, however our averaging has to be done more carefully due to the row-dependent restrictions from \eqref{eq:nested-family}. We define the following estimate to account for them
\begin{equation}\label{eq:deficit-paper}
 \Delta(\ell) \doteq \sum_{i=1}^m(c_i-1)
             =\frac{m(m+1)}2-\sum_{i=1}^m\ell_i.
\end{equation}
This is $m(m-1)/2$ for the unrestricted permutation and $0$ for a unique
assignment, including the empty assignment.

\begin{lemma}[Rank variance under nested lower bounds]\label{lem:variance}
Let $m\ge1$ and $\Sigma$ be uniform on $\mathcal F(\ell)$. Then
\begin{align}
 \sum_i\operatorname{Var}[\Sigma(i)] \geq \frac{\Delta^2}{12m}
 \qquad \text{and} \qquad
\sum_i\operatorname{Var}[\sqrt{r_{\Sigma(i)}}] \geq \frac{\Delta^2}{48md}.
 \label{eq:sqrt-rank-variance}
\end{align}
\end{lemma}
\begin{paperproof}{lem:variance}
Sample $\Sigma$ by filling rows $m,m-1,\ldots,1$, choosing uniformly among
the available allowed columns, as in the proof of \cref{lem:count}.
Given the assignments to rows $i+1,\ldots,m$, $\Sigma(i)$ is uniform on
$c_i$ distinct column \emph{indices}, hence on $c_i$ distinct integers. If these are $a_1<\cdots<a_{c_i}$, then
$a_v-a_u\ge v-u$; the pairwise variance formula shows their variance is
at least that of $1,\ldots,c_i$, namely $(c_i^2-1)/12$.
Total variance and Cauchy--Schwarz give
\[
 \sum_i\operatorname{Var}[\Sigma(i)]
 \ge\frac1{12}\sum_i(c_i^2-1)
 \ge\frac1{12}\sum_i(c_i-1)^2\ge\frac{\Delta^2}{12m}.
\]
Original ranks satisfy $|r_a-r_b|\ge|a-b|$. The square-root comparison in
\cite[Section~3.3, proof of Theorem~1.1]{grimmer2026lower} therefore gives
$|\sqrt{r_a}-\sqrt{r_b}|\ge|a-b|/(2\sqrt d)$.
Apply this to independent copies and use
$\operatorname{Var}(X)=\tfrac12\E(X-X')^2$.
The conditional row argument above handles the nested lower bounds in
$\mathcal F(\ell)$.
\end{paperproof}

\begin{corollary}[Conditional optimization error]\label{cor:conditional-error}
After fixing the seed, at a final transcript $h$ with $m>0$, any fixed vector $x$ feasible
on every compatible instance satisfies
\begin{equation}\label{eq:conditional-error}
 \E[f(x)-f(x_\star^{(\Pi)})\mid h]
 \ge\frac{\rho^2C^2\Delta^2}{48md}.
\end{equation}
\end{corollary}
\begin{proof}
Fix the final transcript $h$ and a point $x$ feasible for every instance
consistent with $h$. Although $x$ is fixed, the optimizer
$X_\star\doteq x_\star^{(\Pi)}$ may differ between those instances. On each compatible instance,
feasibility and constrained first-order optimality give
$f(x)-f(X_\star)\ge\|x-X_\star\|_2^2/2$. Moreover,
\[
 \E[\|x-X_\star\|_2^2\mid h]
 =\|x-\E[X_\star\mid h]\|_2^2
   +\sum_v\operatorname{Var}[X_{\star,v}\mid h].
\]
The first term is nonnegative. Resolved coordinates have zero variance;
by \eqref{eq:optimum} the remaining sum is
$2\rho^2C^2\sum_i\operatorname{Var}[\sqrt{r_{\Sigma(i)}}]$.
Apply \eqref{eq:sqrt-rank-variance} and the factor $1/2$ above.
\end{proof}

\section{Expected decrease in one query}
\label{sec:one-query-progress}

We bound the expected decrease of $\Delta$ in one LMO query by counting
pairs of compatible assignments that differ by a single swap and produce
different answers.

Fix the transcript and the next LMO query direction $p\in\R^d$;
the oracle minimizes $\langle p,z\rangle$ over $S_\Pi$.
If $m=0$ or $p=0$, $\Delta$ is unchanged; assume $m\ge1$ and $p\ne0$.
The unresolved entries of this full vector may still all be zero. We build a graph as follows. Each vertex is a compatible assignment $\sigma\in\mathcal F(\ell)$, where $\ell$ is the bound vector \emph{before} the new query's information is incorporated; in the following we use a superscript $+$ to denote quantities after the query, in particular, $\Delta = \Delta(\ell)$ and $\Delta^+ \doteq \Delta(\ell^+)$. For a row $i$ with column $s\doteq\sigma(i)$, choose
$a\in\{\ell_i,\ldots,s-1\}$ and put $k\doteq\sigma^{-1}(a)$.
Swap these columns: $\sigma'(i)\doteq a$, $\sigma'(k)\doteq s$, leaving
other rows unchanged. The swap is allowed: row $i$ accepts $a$, and row $k$
receives a larger column index. It exchanges original ranks $r_a,r_s$
between the fixed physical coordinates $\iota(i),\iota(k)$.
Each swap gives a directed edge $\sigma\to\sigma'$; both orientations
are counted. Every vertex has the same number of outgoing edges:
\[
  \sum_i(\sigma(i)-\ell_i)
  =\frac{m(m+1)}2-\sum_i\ell_i
  =\Delta.
\]

An \emph{answer class} consists of the assignments producing the same
numerical LMO answer. Applying the earlier indexing construction to the
transcript extended by this answer represents the class as
$\mathcal F(\ell^+)$.

A swap stays within an answer class exactly when it preserves its fixed
assignments and new bounds, hence when it remains an allowed swap of the updated family.
Each vertex has $\Delta^+\le\Delta$ outgoing edges within its class and
$\Delta-\Delta^+$ leaving it. Since $\Sigma$ is uniform on $\mathcal F(\ell)$,
\begin{equation}\label{eq:edge-boundary-paper}
  \E[\Delta-\Delta^+]
  =\frac{\#\{\text{directed swap edges whose endpoints give different answers}\}}
         {|\mathcal F(\ell)|}.
\end{equation}
We provide an example in \cref{app:swap-example}.

To bound this count, we use the threshold equation. The disclosed coordinates
contribute
\[
  \psi(\tau) \doteq \sum_{v\in K_t}(\tau p_v^2-\Pi(v))_+.
\]
In the formulas for the unresolved row (which gives the bounds) we use $p_i$ as an abbreviation for the full component $p_{\iota(i)}$; alternatively think of it as relabeling so that we can apply the same argument on a smaller set. For $\sigma\in\mathcal F(\ell)$, the threshold $\tau_\sigma$
is the unique solution to
\begin{equation}\label{eq:conditional-water-level}
  \psi(\tau)+
  \sum_{i=1}^m(\tau p_i^2-r_{\sigma(i)})_+=1.
\end{equation}
The function $\psi$ is known, nonnegative, nondecreasing and independent of
$\sigma$. Below $\tau_\sigma$, the left side and every summand are below $1$.
A swap exchanges two ranks and we first count swaps whose larger rank has a nonzero LMO component.

\begin{lemma}[Active larger ranks]\label{lem:active-down}
For the fixed query $p$ and a compatible hidden assignment
$\sigma\in\mathcal F(\ell)$, define
\[
  A_\sigma\doteq\{i\in[m]:\tau_\sigma p_i^2>r_{\sigma(i)}\}.
\]
These are the pre-query unresolved rows whose ranks are disclosed by nonzero
components of the new LMO answer. With $\Sigma$ uniform on $\mathcal F(\ell)$,
$A_\Sigma$ denotes the corresponding random set, and
\begin{equation}\label{eq:active-down}
  \E\sum_{i\in A_\Sigma}(\Sigma(i)-\ell_i)\le m.
\end{equation}
\end{lemma}

\begin{proof}
Since $\Sigma$ is uniform, multiplying the expectation in
\eqref{eq:active-down} by $|\mathcal F(\ell)|$ counts triples
$(\sigma,i,a)$ with $i\in A_\sigma$ and
$\ell_i\le a<s\doteq\sigma(i)$. Map each triple to
$(\sigma',s)\in\mathcal F(\ell)\times\{1,\ldots,m\}$, where $\sigma'$
is the compatible assignment obtained by swapping $a,s$. We show injectivity.

Suppose for contradiction distinct rows $i,k$ give the same target $(\sigma',s)$.
Their columns $a_i\doteq\sigma'(i)$ and $a_k\doteq\sigma'(k)$ are distinct
and below $s$. In the preimage where row $i$ carries $r_s$, activity gives
$\tau_\sigma p_i^2>r_s>0$, hence $p_i\ne0$; the other preimage similarly
gives $p_k\ne0$.
At $\tau=r_s/p_i^2<\tau_\sigma$, the left side of
\eqref{eq:conditional-water-level} is below $1$, hence
$\frac{p_k^2}{p_i^2}<\frac{r_{a_k}+1}{r_s}$.
The other preimage gives $p_i^2/p_k^2<(r_{a_i}+1)/r_s$.
Multiplying these two ratio bounds gives
\[
  1=\frac{p_k^2}{p_i^2}\frac{p_i^2}{p_k^2}
  <\frac{(r_{a_k}+1)(r_{a_i}+1)}{r_s^2}<1,
\]
where the last inequality uses that $r_{a_i},r_{a_k}$ are distinct integers
below $r_s$. This contradiction proves the row is unique;
$a=\sigma'(i)$ and the inverse swap recover $\sigma$.
There are at most $m|\mathcal F(\ell)|$ targets.
\end{proof}

An inactive row $i$ has $\tau_\sigma p_i^2\le r_{\sigma(i)}$.
Its zero LMO component excludes a previously allowed smaller rank $r_a$
exactly when
\begin{equation}\label{eq:certified-cell-definition}
 \ell_i\le a<\sigma(i),\qquad
 r_a<\tau_\sigma p_i^2\le r_{\sigma(i)}.
\end{equation}
Let $N(\sigma)$ count these pairs $(i,a)$. The condition $\ell_i\le a$
ensures that the rank was allowed before the answer.

\begin{lemma}[Reduction to newly excluded pairs]\label{lem:certified-paper}
For the count $N(\sigma)$ defined by \eqref{eq:certified-cell-definition},
\begin{equation}\label{eq:boundary-certified}
  \E[\Delta-\Delta^+]\le2m+2\E N(\Sigma).
\end{equation}
\end{lemma}

\begin{paperproof}{lem:certified-paper}
By \cref{lem:active-down}, edges whose larger rank is active at either
endpoint contribute at most $2m$ in expectation: apply the bound at each
endpoint, reversing orientations for the second.

Now consider an edge swapping $r_a<r_s$, with the larger-rank row inactive
at both endpoints. Let $\Phi,\Phi'$ be their threshold left sides in
\eqref{eq:conditional-water-level}, with roots $\tau,\tau'$.
Suppose neither endpoint's larger-rank row excludes $r_a$.
At $\tau$, replacing $r_s$ by $r_a$ creates no positive term, and replacing
$r_a$ by $r_s$ in the other row cannot increase its term. Hence
$\Phi'(\tau)\le\Phi(\tau)=1$ and $\tau'\ge\tau$.
Reversing the swap gives $\Phi(\tau')\le\Phi'(\tau')=1$ and $\tau\ge\tau'$.

The thresholds agree; both swapped rows are inactive, and all other
assignments are unchanged. The LMO answers therefore coincide by
\cref{lem:exact-lmo}. Thus every remaining edge between different answers
has a pair $(i,a)$ counted by $N(\sigma)$ at some endpoint $\sigma$.
The endpoint and pair determine the swap, so at most two orientations are
counted per pair. This gives the remainder $2\E N(\Sigma)$.
\end{paperproof}

It remains to establish a linear bound on $\E N(\Sigma)$, which we will do via the threshold equation. The right-hand side of \eqref{eq:conditional-water-level} is $1$, and all
terms are nonnegative. Thus for every unresolved row $j$ with $p_j^2>0$,
\begin{equation}\label{eq:shifted-occupied}
 \tau_\sigma\le\frac{r_{\sigma(j)}+1}{p_j^2}.
\end{equation}
Rows with $p_j=0$ are omitted when minimizing the quantities on the right.
The nonnegative known-coordinate contribution $\psi$ does not change this
inequality.

\begin{lemma}[Linear bound on newly excluded pairs]
\label{lem:certified-linear}
For every conditional distribution of the form \eqref{eq:nested-family} and
every query,
\[
 \E N(\Sigma)\le2m.
\]
\end{lemma}

\begin{proof}
Separate the pairs in \eqref{eq:certified-cell-definition} according to
whether $r_a+1<\tau_\sigma p_i^2$. The pairs not satisfying this strict
inequality have $\tau_\sigma p_i^2-1\le r_a<\tau_\sigma p_i^2$.
This half-open interval contains at most one integer, so there is at most one
such pair per row. This includes equality at the left endpoint and requires
no generic-position hypothesis.

For a pair satisfying $r_a+1<\tau_\sigma p_i^2$, set $s\doteq\sigma(i)>a$ and
transpose indices $a,s$ to obtain $\sigma'$. The target is compatible because
row $i$ accepts $a$ and the other row moves upward. Map the triple
$(\sigma,i,a)$ to $(\sigma',s)$.

Here $p_i^2>0$ and $u\doteq\frac{r_a+1}{p_i^2}<\tau_\sigma$.
By \eqref{eq:shifted-occupied}, every quantity
$(r_{\sigma(j)}+1)/p_j^2$ with $p_j\ne0$ is strictly greater than $u$.
After the transposition, row $i$ has value $u$. The value of the other swapped
row increases if its query component is nonzero, and all remaining values are
unchanged. Thus $i$ is the unique minimizer of
$(r_{\sigma'(j)}+1)/p_j^2$ over nonzero unresolved query components. The target
$\sigma'$ determines $i$ and $a=\sigma'(i)$; with $s$, it recovers $\sigma$
by transposing back. This proves injectivity without using any property of
the new threshold $\tau_{\sigma'}$.

There are at most $m|\mathcal F(\ell)|$ targets, so the expected number of
pairs satisfying the strict inequality is at most $m$. Adding the at most
one remaining pair per row gives $2m$. If all unresolved query components
are zero, no pairs satisfy \eqref{eq:certified-cell-definition}; if $m=0$,
the count is also zero. The zero full query has already been treated by the
oracle convention.
\end{proof}

\begin{corollary}[One-query decrease]\label{cor:one-query}
For every adaptive transcript,
\begin{equation}\label{eq:one-query}
 \E[\Delta-\Delta^+\mid\mathcal H_t]\le6m.
\end{equation}
\end{corollary}

\section{Conclusion}
\label{sec:conclusion}
We established a lower bound on the objective error achievable by unrestricted randomized first-order LMO methods on the family of instances from \citet{grimmer2026lower}, with final outputs feasible on every instance and seed. The main step is to control information from
arbitrary adaptive numerical answers: exact posterior constraints, optimizer
variance and a swap-graph progress bound combine to give expectation,
constant-probability and stopping-time results. The geometric regime remains that of the inherited construction from \citet{grimmer2026lower}. The result does not settle the optimal dependence on independently fixed curvature and diameter as we discussed in \Cref{app:parameter-comparison}. Separately, our trajectory analysis describes local sensitivity near the prescribed slow Frank--Wolfe segments from \citet{halbey2026lower} and a measure-zero set of starts satisfying a particular eventual slow envelope. Neither contradicts the prior finite-horizon lower bound or proves an almost-sure faster rate for random starts.

\label{paper-main-end}

\section*{AI use statement}
Generative AI tools assisted throughout the research process with formulation and revision of mathematical arguments, implementation of verification and plotting scripts, formalization, and manuscript editing. The associated checks include symbolic calculations, numerical and adversarial tests, and machine-checked proofs for the results covered by the formalization. All claims have been independently verified by the authors. We take responsibility for the final content of this work, including text, claims or artifacts produced with the aid of generative AI.

\section*{Acknowledgements}
This research was conducted in the context of the Agentic AI for Mathematics (EF-LiOpt-3) project at the Berlin Mathematics Research Center MATH$^+$ (EXC-2046/2, project ID 390685689), funded by the Deutsche Forschungsgemeinschaft (DFG, German Research Foundation) under Germany's Excellence Strategy. This work has been partially supported by OpenAI's Academic Researchers program.  \bibliographystyle{plainnat}

\clearpage
\appendix
\PrintDeferredProofs
\section{Detailed comparison with existing lower bounds}
\label{app:existing-bounds}

A lower bound's implications for randomization depend on its algorithm
class, its horizon-dependent data, and the time at which the instance is fixed
relative to the seed. We examine these choices in the two recent $1/T^2$
lower bounds.

\subsection{Distributional and randomized values}\label{sec:minimax}
For a finite instance family $\mathcal I$, let $\mathcal A_T$ be the
admissible deterministic methods using at most $T$ calls. For bounded
measurable loss $\mathcal L\in[0,M_{\rm loss}]$, define
\begin{align}
 V_{\rm dist}(\mu)&\doteq\inf_{A\in\mathcal A_T}\E_{I\sim\mu}\mathcal L(A,I),
 \label{eq:distributional-value}\\
 V_{\rm rand}&\doteq\inf_\nu\sup_{I\in\mathcal I}\E_{A\sim\nu}\mathcal L(A,I),
 \label{eq:randomized-value}\qquad
 V_{\rm det}\doteq\inf_{A\in\mathcal A_T}\sup_I\mathcal L(A,I).
\end{align}
The deterministic method may know $\mu$; $\nu$ ranges over arbitrary
measurable mixtures of methods. In our application, compactness, feasibility
and finiteness give
$M_{\rm loss}\doteq\max_\pi\max_{x\in S_\pi}(f(x)-f(x_\star^{(\pi)}))<\infty$.
\begin{proposition}[Yao's principle]\label{prop:minimax}
For every prior $\mu$, and under closure under arbitrary mixtures,
\begin{align}
 V_{\rm dist}(\mu)&\le V_{\rm rand}\le V_{\rm det},\label{eq:yao-inequality}\\
 \sup_\mu V_{\rm dist}(\mu)&=V_{\rm rand}.\label{eq:minimax-equality}
\end{align}
\end{proposition}
\begin{paperproof}[inline]{prop:minimax}
For any $\mu,\nu$, boundedness and measurability give
$\sup_I\E_\nu\mathcal L\ge\E_\nu\E_\mu\mathcal L\ge V_{\rm dist}(\mu)$.
Take the infimum over $\nu$; deterministic methods are degenerate mixtures.
For equality, take the closure of the deterministic loss vectors in the
finite-dimensional cube $[0,M_{\rm loss}]^{|\mathcal I|}$ and its compact
convex hull. Closure does not change the infimum of a linear functional,
and finite mixtures approximate its points. The finite-dimensional minimax
theorem gives the equality \citep{yao1977probabilistic,sion1958minimax}.
\end{paperproof}
A lower bound for one fixed prior therefore suffices.

\subsection{Further related work}\label{app:other-related-work}
Classical and information-theoretic oracle bounds limit what a query reveals
\citep{nemirovski1983problem,raginsky2009information,agarwal2012information,
braun2015information,braun2017lower}. Our approach is somewhat similar but it determines the exact conditional distribution (and provides a combinatorial statistic $\Delta$) rather than converting real-valued answers into a finite alphabet whose information can be measured with information-theoretic notions such as entropy. Adaptive permutation problems such as permutation Mastermind have related conditional-assignment questions
\citep{larsen2024mastermind}, but the answers of an LMO require rather different threshold constraints. There is also a recent result by \citet{hazan2026linearly} which establishes linear-convergence for smooth convex sets, however it uses membership-based supporting-tangent computations and a uniform interior rolling-ball
condition; both the oracle and the geometry differ
from the nonsmooth-boundary sets considered here.

\subsection{\citep{grimmer2026lower}: deterministic and randomized quantifiers}
\label{sec:grimmer-liu-gap}

\citet{grimmer2026lower} prove a sharp $\Omega(L\diam(S)^2/T^2)$ rate for so-called \emph{LMO-span methods}, which is a subclass of general LMO-based methods. In the case of \emph{deterministic methods}, this bound can then be converted into one for \emph{general} LMO-based methods using a resisting oracle; see further below.

More precisely, \citet{grimmer2026lower} show that for \emph{LMO-span methods}, one fixed ordered instance is hard for every admissible trajectory. Thus, fixing a
randomized method's seed $U=u$ gives another admissible trajectory on that
same instance. The lower bound therefore holds for every $u$ and remains
valid after averaging over $U$. Put differently, for randomized LMO-span methods, \citet{grimmer2026lower} already implicitly provide a lower bound on the rate of $\Omega(L\diam(S)^2/T^2)$ however not for general randomized LMO-based methods. We will now show that the LMO-span restriction is indeed significant; see also the discussion in \citet{grimmer2026lower}.

\begin{remark}[Methods included in and excluded from the span class]
Let $z_{k+1}\doteq\LMO_S(p_k)$. The \emph{LMO-span} model of \citet{grimmer2026lower} requires
\begin{align*}
 p_k&\in\spann\!\left(
   \{x_\ell-x_0\}_{\ell=1}^k\cup
   \{z_\ell-x_0\}_{\ell=1}^k\cup
   \{\nabla f(x_\ell)\}_{\ell=0}^k\right)\setminus\{0\},\\
 x_{k+1}&\in\operatorname{conv}\{x_0,z_1,\ldots,z_{k+1}\}.
\end{align*}
There is no computational restriction within these sets. The class contains
ordinary Frank--Wolfe with open-loop steps, short steps, or exact line search,
and the usual away-step, pairwise, fully corrective, and blended pairwise
variants: their LMO direction is the gradient or a linear combination of
observed vectors, and their feasible iterates remain in the hull of previously
obtained atoms. Randomizing their step choices preserves these conditions for
every fixed seed. 

The broader class of general first-order LMO-based methods also permits directions such as
\[
  p_k\doteq\nabla f(x_k)+\delta_k e_{J_k}
  \qquad\text{or}\qquad
  p_k\doteq H_k\nabla f(x_k),
\]
which the span condition excludes when the added coordinate or transformed
gradient lies outside the observed span. Such directions can still be used
with the feasible update
$x_{k+1}\doteq(1-\gamma_k)x_k+\gamma_k z_{k+1}$.

For the following examples, $C,w_k,S_\pi,M$, and $f$ are the family parameters
in \eqref{eq:constants}--\eqref{eq:objective}, and
$x_\star^{(\pi)}$ is the optimizer in \eqref{eq:optimum}. At $x_0=0$,
$\nabla f(0)=-M\mathbf1$, so a random coordinate perturbation already lies
outside the permitted span at the first query. The span class also excludes
methods that probe the hidden set and then construct a feasible output outside
the atom hull. Indeed,
\[
  \LMO_{S_\pi}(e_i)
  =-\left(\sqrt{w_{\pi(i)}^2+2C^2}-w_{\pi(i)}\right)e_i,
\]
and its magnitude is strictly decreasing in $w_{\pi(i)}$. An exact coordinate
query thus identifies $\pi(i)$. After $d$ coordinate queries, a method can
recover $\pi$ and output its explicit optimizer. The query atoms are
componentwise nonpositive, whereas the optimizer has $d-1$ positive
coordinates, so that output is outside their convex hull.

There is also a one-call separation on a specified public instance. Fix
$S_{\pi_0}$ and its optimizer $x_\star^{(\pi_0)}$. A method can store
$p^\star\doteq\nabla f(x_\star^{(\pi_0)})$, query
$z_1\doteq\LMO_S(p^\star)$, and output $z_1$. This output is feasible on every set
presented to the method. On $S_{\pi_0}$, first-order optimality gives
\[
 \langle p^\star,x-x_\star^{(\pi_0)}\rangle\ge0
 \quad\text{for every }x\in S_{\pi_0}.
\]
Uniqueness of linear minimization over the strongly convex set gives
$z_1=x_\star^{(\pi_0)}$, so the method has zero error on that instance. It is
not a LMO-span method though: $p^\star$ has unequal coordinates, whereas the initial permitted span is $\spann\{\nabla f(0)\}=\spann\{\mathbf1\}$.

These are separations between algorithm classes on specified instances, not
$o(T^{-2})$ worst-case algorithms. \citet{grimmer2026lower}'s resisting oracle theorem rules out such an improvement for deterministic methods, and our distributional theorem rules it out for randomized methods. 
\end{remark}

To go from LMO-span methods to unrestricted deterministic LMO-based methods, \citet{grimmer2026lower} instead assign the
hidden coordinate ranks gradually in response to the queries. After producing
a deterministic transcript, they complete those assignments to a fixed
permutation that would have generated the same transcript. This is a
resisting-oracle argument.

To make this distinction clear, write $\mathcal A_{\rm det}$ for the deterministic class, $I$ for an instance, $\mathcal L$ for the loss, and $c_T$ for the asserted lower threshold. These symbols are schematic here; see also \cref{sec:minimax}. The resisting
argument gives
\begin{equation}\label{eq:deterministic-quantifiers}
  \forall A\in\mathcal A_{\rm det}\quad
  \exists I_A\quad
  \mathcal L(A,I_A)\ge c_T.
\end{equation}
Applying it separately to fixed seeds yields only $\forall u\,\exists I_u$.
A randomized lower bound requires
\begin{equation}\label{eq:randomized-quantifiers}
  \forall\nu\in\Delta(\mathcal A_{\rm det})\quad
  \exists I\quad
  \E_{A\sim\nu}\mathcal L(A,I)\ge c_T,
\end{equation}
where $\Delta(\mathcal A_{\rm det})$ denotes probability distributions over
the deterministic methods. The same instance must work after averaging over
seeds and cannot depend on the realized seed; this is precisely what makes things much harder.

A sufficient distributional statement is
\begin{equation}\label{eq:distributional-quantifiers}
  \exists\mu\quad\forall A\in\mathcal A_{\rm det}:
  \qquad
  \E_{I\sim\mu}\mathcal L(A,I)\ge c_T,
\end{equation}
with $\mu$ fixed before the method runs, i.e., the algorithm may know the distribution but not the realization from that distribution. Yao's inequality then gives \eqref{eq:randomized-quantifiers}. We take $\mu$ uniform over the coordinate
permutations in \citet{grimmer2026lower}'s family and we retain their sets $S_\pi$,
square-root-spaced weights, common quadratic objective, optimizer and
exact-LMO formulas, and conversion of optimizer variation into objective error. However, their resisting argument averages over completions of an unconstrained
suffix of ranks, while under our fixed prior, an unrestricted query can instead
(i) reveal several ranks simultaneously and (ii) impose row-dependent lower bounds that need to be taken into consideration for the conditional distributions. In \Crefrange{sec:exact-transcripts}{sec:one-query-progress} we characterize these
conditional distributions, prove their variance bound, and bound the expected
decrease of $\Delta$. We emphasize, that our argument is a distributional complexity argument with a fixed prior for unrestricted randomized methods and \emph{not} a randomization of the resisting oracle, which would most likely reduce to their randomized span-class methods lower bound.

\subsection{\citep{halbey2026lower}: prescribed trajectories and initialization sensitivity}
\label{sec:halbey-diagnosis}

\citet{halbey2026lower} study
\begin{equation}\label{eq:halbey-model}
  \min_{\|x\|_2\le1}\|x-p\|_2^2,
  \qquad \|p\|_2=1,
\end{equation}
for Frank--Wolfe with exact line search or short steps. For every requested horizon, they construct an
initialization whose objective gap is $\Theta(1/t^2)$ over a long segment.
This proves sharpness of the $O(1/T^2)$ rate for those updates.
It does not determine the measure of such initializations or lower-bound
unrestricted LMO methods. 

We analyze the sensitivity of their exact-line-search backward construction (and hence also short-steps). For the exact-line-search segments analyzed below,
the restriction is the slow-envelope condition
\eqref{eq:halbey-slow-envelope} on the contraction factors at every step;
we make no sensitivity claim for portions of a trajectory outside that condition. For a nonoptimal iterate with noncollinear residual, write
\[
  e\doteq x-p,\qquad r\doteq\|e\|_2,\qquad u\doteq e/r,
  \qquad \theta\doteq\langle u,p\rangle.
\]
Here $r$ is the distance to the optimizer and $\theta$ specifies the residual's
orientation. Rotational symmetry keeps the trajectory in the plane spanned
by $x_0$ and $p$. The unit-ball LMO returns $v=-u$. Exact line search along
$d\doteq v-x=-(1+r)u-p$ gives
\[
  \gamma(r,\theta)
  =\frac{r(1+r+\theta)}{(1+r)^2+1+2(1+r)\theta}.
\]
Substitution in $x^+\doteq x+\gamma d$ yields the exact two-dimensional recurrence
\begin{equation}\label{eq:halbey-polar-map}
  F(r,\theta)\doteq(r^+,\theta^+)
  =\bigl(rs(r,\theta),-(1+r)s(r,\theta)\bigr),
  \qquad
  s(r,\theta)^2
  =\frac{1-\theta^2}{(1+r)^2+1+2(1+r)\theta}.
\end{equation}

Its Jacobian describes the first-order propagation of perturbations in these
two state coordinates.

On the slow segments constructed by \citet{halbey2026lower}, the contraction
$s_t\doteq r_{t+1}/r_t$ satisfies
\begin{equation}\label{eq:halbey-slow-envelope}
  s_t=1-\frac43r_t+c_tr_t^2,
  \qquad 1\le c_t\le\frac52.
\end{equation}
The radial recurrence gives $r_t=\Theta(1/t)$. After one forward step, the
angular recurrence gives
\begin{equation}\label{eq:halbey-slow-curve}
  \theta_t=-1+\frac13r_t+O(r_t^2).
\end{equation}
Along this curve, differentiation of \eqref{eq:halbey-polar-map} gives
\begin{equation}\label{eq:halbey-jacobian-limit}
  DF(r_t,\theta_t)
  =\begin{pmatrix}1&0\\[1mm]1&-2\end{pmatrix}+O(r_t).
\end{equation}

The limiting matrix arising from $r \to 0^+$ with $\theta_t = -1+\frac13r_t+O(r_t^2) \to -1$ along the trajectory has tangent eigenvector $(3,1)^\top$ with eigenvalue $1$ and transversal eigenvector $(0,1)^\top$ with eigenvalue $-2$, i.e., it leaves $(3,1)$ invariant, tangent to \eqref{eq:halbey-slow-curve}, and multiplies the transverse direction $(0,1)$ by $-2$. In particular, a transverse perturbation thus changes sign and approximately doubles in one slow step. The following statement accounts for the variation of the
Jacobians along the segment.

\begin{proposition}[Initialization sensitivity along a slow segment]
\label{prop:halbey-bit-law}
Reindex a slow segment satisfying \eqref{eq:halbey-slow-envelope} after entry
into a sufficiently small fixed neighborhood of $(0,-1)$ and after one
forward step, so that \eqref{eq:halbey-slow-curve} holds from time zero.
Fix entry scales $0<r_-\le r_+$ and assume $r_-\le r_0\le r_+$, with $r_+$
sufficiently small. Put
\[
  \mathcal J_T
  \doteq DF(r_{T-1},\theta_{T-1})\cdots DF(r_0,\theta_0),
  \qquad \sigma_T\doteq\|\mathcal J_T\|_2.
\]
There are constants $a,A,C>0$, depending only on the fixed entry scales and
slow-envelope bounds, not on the segment length, such that
\begin{equation}\label{eq:halbey-one-bit-law}
  a\,2^T(T+1)^{-C}
  \le \sigma_T
  \le A\,2^T(T+1)^C.
\end{equation}
Consequently $\log_2\sigma_T=T+O(\log T)$.
\end{proposition}
\begin{proof}[Proof Sketch.]
The full proof can be found in \cref{app:halbey-precision}. In the eigenbasis of the
limiting matrix, the Jacobian is $\operatorname{diag}(1,-2)+E_t$, with
$\|E_t\|=O(1/t)$. A cone about the transverse direction is invariant.
Multiplying its expansion factors $2(1-O(1/t))$ gives the lower bound;
harmonic summation contributes only a polynomial factor. An operator-norm
estimate gives the upper bound.
\end{proof}

To state the associated linearized precision requirement, define
\[
  \Lambda_T
  \doteq\max_{0\le t\le T}\|\mathcal J_t\|_2. \qquad \text{and} \qquad \mathcal J_0 = I.
\]
A radius-$u$ uncertainty ball about the initialization has worst-case
first-order trajectory error $u\Lambda_T$. For tolerance $\eta$, the number
of binary digits required in this linearized uncertainty model is
\begin{equation}\label{eq:halbey-float-bits}
  B_{\rm lin}(T,\eta)
  \doteq\left\lceil\log_2\frac{\Lambda_T}{\eta}\right\rceil
  =T+\log_2(1/\eta)+O(\log T).
\end{equation}

For fixed or polynomially shrinking tolerance, the leading term is one binary
digit per slow iteration. The approximately $10^{-300}$ numerical
accuracy reported in \citet{halbey2026lower} for a $1000$-step construction is consistent with this scale: $300\log_2 10 \approx 996.6$. Note, that \Cref{prop:halbey-bit-law} is a worst-case \emph{linearized} statement about small transverse
errors along a prescribed trajectory. It does not imply that every
lower-precision run fails and in particular, a specially representable initialization, a tangential perturbation, or favorable rounding cancellation may behave better. \Cref{app:halbey-precision} proves a separate nonlinear bound
for additive errors in the polar recurrence. 
The exact map also gives a measure-zero result in the plane for the set $\mathcal E$ of initializations whose trajectories never reach $p$ in finite time and
eventually satisfy the slow-envelope condition. These trajectories approach $p$ asymptotically. The result concerns infinite trajectories, not hardness at a single prescribed horizon.

\begin{theorem}[Zero area for the eventual slow envelope]
\label{thm:halbey-measure-zero}
Consider exact-line-search Frank--Wolfe on the unit disk in $\R^2$ for
$f(x)\doteq\|x-p\|_2^2$, with $\|p\|_2=1$, as in \eqref{eq:halbey-model}.
The optimizer $p$ is kept fixed once reached.
Let $\mathcal E$ be the set of $x_0\in B_1(0)$ whose iterates never reach $p$
and satisfy \eqref{eq:halbey-slow-envelope} for every sufficiently large $t$.
Then $\mathcal E$ has two-dimensional Lebesgue measure zero, although it is
uncountable. For every $x_0\in\mathcal E$,
\begin{equation}\label{eq:halbey-eventual-rate}
  t r_t\longrightarrow\frac34,
  \qquad
  t^2\bigl(f(x_t)-f(p)\bigr)\longrightarrow\frac9{16}.
\end{equation}
In particular, an initialization distribution absolutely continuous with
respect to area assigns probability zero to $\mathcal E$.
\end{theorem}
\begin{proof}[Proof Sketch]
The proof can be found in \cref{app:halbey-measure-zero}. With $q\doteq1+\theta$,
the map is injective on
\[
  W\doteq\{(r,q):0<r\le1/100,\ r/4\le q\le2r/5\}
\]
and expands $(r,q)$ area by a factor greater than $6/5$. The area of starts
remaining there for $n$ updates is therefore at most
$\operatorname{area}(W)(5/6)^n$. The envelope forces eventual membership in
$W$, and finite preimages and Cartesian coordinate changes preserve nullity.
A nested-interval argument supplies such a trajectory at every initial
radius in $(0,1/100]$, so the statement is not vacuous.
\end{proof}

For exact line search, or equivalently short steps on this instance, an initial point sampled from a distribution with a density on the disk almost surely lies outside $\mathcal E$ and therefore does not generate an infinite trajectory satisfying the eventual slow-envelope condition. We establish the existence of such trajectories, together with the measure-zero statement, in \cref{thm:halbey-measure-zero}. The stated guarantee in \citet{halbey2026lower} is finite-horizon, with an initialization that may depend on the horizon. This does not rule out long finite slow phases, and \citet{halbey2026lower}'s finite-horizon lower bound is independent of this. Nor does the result establish that every initialization whose trajectory has $\Omega(t^{-2})$ error belongs to $\mathcal E$.

Neither the local sensitivity estimate nor \cref{thm:halbey-measure-zero}
determines the measure of all hard initializations.
Let $T(x_0,\eps)$ be the first time the objective gap is at most $\eps$. For
fixed $\kappa>0$, define
\begin{equation}\label{eq:halbey-hard-set}
  H_\eps(\kappa)
  \doteq\left\{x_0\in B_1(0):T(x_0,\eps)\ge\frac{\kappa}{\sqrt\eps}\right\},
  \qquad
  q_\eps(\kappa)
  \doteq\frac{\operatorname{vol}(H_\eps(\kappa))}
         {\operatorname{vol}(B_1(0))}.
\end{equation}
The finite-horizon strict inequalities define a relatively open set by
continuity. Thus a strict hard point gives a positive-measure subset of
$H_\eps(\kappa)$ and $q_\eps(\kappa)>0$ at the relevant fixed accuracies.
This supplies no lower bound uniform as $\eps\downarrow0$. Independent
uniform starts sampled until one lies in $H_\eps(\kappa)$ have a geometric
waiting time with mean $1/q_\eps(\kappa)$. A limit
$q_\eps(\kappa)\to0$ would therefore imply a diverging expected number of
trials. That limit is unknown: in particular, a positive limit superior or nonconvergence has not been ruled out. 

The inverse recurrence explains why the local sensitivity and nonlinear
perturbation bounds do not settle this measure question. A valid
nonterminating next state $(R,\Theta)$, with $R>0$, determines
$s=-\Theta-R$ and $r=R/s$, but the preceding angle may have two values:
\begin{equation}\label{eq:halbey-two-preimages}
  \theta_\pm
  =-s^2(1+r)
   \pm\sqrt{(1-s^2)\bigl(1-(1+r)^2s^2\bigr)}.
\end{equation}
There can be as many as $2^T$ histories after $T$ inverse steps, whereas the
width of one slow inverse branch is $2^{-T}T^{O(1)}$. Schematically,
\[
  \underbrace{2^T}_{\text{possible branches}}
  \times
  \underbrace{2^{-T}}_{\text{width per branch}}
  \asymp1.
\]
The elementary doubling map exhibits this cancellation. For the
Frank--Wolfe map, a proof that $q_\eps(\kappa)\to0$ would additionally need a
global estimate on escape or on the measures of the inverse histories that
remain hard.

\subsection{Proofs for the \citep{halbey2026lower} calculation}
\label{app:halbey-precision}

In this appendix we derive the exact recurrence and prove the variable-Jacobian
estimate in \cref{prop:halbey-bit-law}. We then consider additive
nonlinear perturbations and prove \cref{thm:halbey-measure-zero}
for the eventual slow envelope. Some of the computations are already contained in \citep{halbey2026lower} while others are relatively straightforward but tedious. We include both in detail for the sake of completeness.

\subsubsection{Exact polar dynamics}

Fix $\|p\|_2=1$ and $f(x)\doteq\|x-p\|_2^2$ on the unit ball. For $x\ne p$, write
\[
  e\doteq x-p=ru,
  \qquad r\doteq\|e\|_2,
  \qquad \theta\doteq\langle u,p\rangle.
\]
The LMO minimizes $\langle e,v\rangle$ over $\|v\|_2\le1$ and returns
$v=-u$. Its displacement from the current iterate is
\[
  d\doteq v-x=-u-(p+ru)=-(1+r)u-p.
\]
Along $x+\gamma d$, exact line search minimizes $\|e+\gamma d\|_2^2$.
The unconstrained minimizer is $-\langle e,d\rangle/\|d\|_2^2$, and
\[
  \langle e,d\rangle=-r(1+r+\theta),
  \qquad
  D(r,\theta)\doteq\|d\|_2^2
  =(1+r)^2+1+2(1+r)\theta.
\]
Thus
\begin{equation}\label{eq:app-gamma}
  \gamma(r,\theta)=\frac{r(1+r+\theta)}{D(r,\theta)}.
\end{equation}
For the noncollinear feasible states in the slow construction this step-size belongs to $(0,1)$, so clipping to the Frank--Wolfe segment has no effect and we can ignore it.

Set $e^+\doteq e+\gamma d$, $r^+\doteq\|e^+\|_2$, and
$s\doteq r^+/r$. Substitution gives
\[
  s^2=1-\frac{\langle e,d\rangle^2}{r^2\|d\|_2^2}
  =1-\frac{(1+r+\theta)^2}{D(r,\theta)}
  =\frac{1-\theta^2}{D(r,\theta)}.
\]
Exact line search also makes $e^+$ orthogonal to $d$, so
\[
  0=\langle e^+,d\rangle
   =-(1+r)\langle e^+,u\rangle-
     \langle e^+,p\rangle.
\]
Direct substitution of \eqref{eq:app-gamma} gives
$\langle e^+,p\rangle=-r(1+r)s^2$. Dividing by $r^+=rs$ gives
$\theta^+=-(1+r)s$. Hence the exact recurrence is
\[
  F(r,\theta)=\bigl(rs,-(1+r)s\bigr),
  \qquad
  s^2=\frac{1-\theta^2}{D(r,\theta)},
\]
as used in \cref{sec:halbey-diagnosis}.

\subsubsection{The slow curve and its Jacobian}

Assume the slow-envelope hypothesis from \citet{halbey2026lower}
\[
  s_t=1-\frac43r_t+c_tr_t^2,
  \qquad 1\le c_t\le\frac52.
\]
Since $r_{t+1}=r_ts_t$,
\begin{equation}\label{eq:app-reciprocal-increment}
  \frac1{r_{t+1}}-\frac1{r_t}
  =\frac{1-s_t}{r_ts_t}
  =\frac{4/3-c_tr_t}{s_t}.
\end{equation}
For sufficiently small fixed radii, the last expression is bounded above and
below by positive constants uniformly over $c_t\in[1,5/2]$. After reindexing
at entry into this neighborhood, summation gives
$r_t=\Theta((t+1)^{-1})$, with the fixed entry-scale bounds from the
proposition.

The angular recurrence gives
\begin{align*}
  \theta_{t+1}
  &=-(1+r_t)s_t\\
  &=-1+\frac13r_t+
    \left(\frac43-c_t\right)r_t^2-c_tr_t^3.
\end{align*}

Since $r_{t+1}=r_t+O(r_t^2)$, after one forward step this is
\[
  \theta_{t+1}=-1+\frac{1}{3} r_{t+1} + O(r_{t+1}^2).
\]
The forward recurrence therefore determines the coefficient $1/3$.

We now differentiate at a generic state $(r,\theta)$, where
$r=\|x-p\|_2>0$ and $\theta=\langle(x-p)/r,p\rangle\in(-1,1)$.
Recall the scalar denominator and write the contraction factor as a function
of these coordinates:
\begin{align*}
  D(r,\theta)&\doteq(1+r)^2+1+2(1+r)\theta=\|d\|_2^2,\\
  s(r,\theta)&\doteq\sqrt{\frac{1-\theta^2}{D(r,\theta)}}.
\end{align*}
Along the trajectory, $s_t=s(r_t,\theta_t)$. Write
$s_r\doteq\partial s/\partial r$ and
$s_\theta\doteq\partial s/\partial\theta$, holding the other coordinate fixed.
Taking logarithms gives
\[
  2\log s=\log(1-\theta^2)-\log D.
\]
Differentiating this identity with respect to $r$ and $\theta$ separately,
using $\partial_r\log s=s_r/s$ and
$\partial_\theta\log s=s_\theta/s$, gives
\begin{align}
  s_r
  &=-s\frac{1+r+\theta}{D},
  \label{eq:app-s-r}\\
  s_\theta
  &=-s\left(
      \frac{\theta}{1-\theta^2}+\frac{1+r}{D}
    \right).
  \label{eq:app-s-theta}
\end{align}
Thus
\begin{equation}\label{eq:app-full-jacobian}
  DF(r,\theta)
  =\begin{pmatrix}
      s+rs_r & rs_\theta\\
      -s-(1+r)s_r & -(1+r)s_\theta
    \end{pmatrix}.
\end{equation}
For the slow-curve substitution, set
$\beta\doteq(\theta+1-r/3)/r^2$; the preceding forward-step expansion gives a
fixed bound $|\beta|\le B$. 
Substituting $\theta=-1+r/3+\beta r^2$ in
\eqref{eq:app-s-r}--\eqref{eq:app-s-theta} yields
\[
  s=1+O(r),
  \qquad s_r=-2+O(r),
  \qquad s_\theta=2+O(r).
\]
Here all three remainders are bounded by $Cr$ in absolute value for
$0<r\le\delta$ and $|\beta|\le B$, with $C,\delta>0$ depending only on $B$.
Consequently,
\begin{equation}\label{eq:app-jacobian-limit}
  DF(r_t,\theta_t)
  =J_\infty+O(r_t),
  \qquad
  J_\infty\doteq\begin{pmatrix}1&0\\1&-2\end{pmatrix}.
\end{equation}
Its eigenvectors satisfy
\[
  J_\infty(3,1)^\top=(3,1)^\top,
  \qquad
  J_\infty(0,1)^\top=-2(0,1)^\top.
\]
The first is tangent to $\theta=-1+r/3$; the second changes the angle while
fixing the residual radius $r$ to first order.

\subsubsection{Products of the varying Jacobians}

\begin{proof}[Proof of \cref{prop:halbey-bit-law}]
\Cref{eq:app-reciprocal-increment} gives
$r_t=\Theta((t+1)^{-1})$. By \eqref{eq:app-jacobian-limit},
\[
  J_t\doteq DF(r_t,\theta_t)
  =J_\infty+O((t+1)^{-1}).
\]
Let
\[
  P\doteq\begin{pmatrix}3&0\\1&1\end{pmatrix}
\]
be the fixed eigenvector matrix of the limiting matrix $J_\infty$ and observe that $P^{-1}J_\infty P=\operatorname{diag}(1,-2)$. There are fixed $K,t_0>0$ such that
\[
  P^{-1}J_tP=\operatorname{diag}(1,-2)+E_t,
  \qquad
  |(E_t)_{ij}|\le\epsilon_t
  \doteq\frac{K}{t+t_0}.
\]
Taking a smaller fixed entry neighborhood if needed increases $t_0$ and
ensures $\epsilon_t\le1/4$ at each reindexed step.

Let us now write a linearized perturbation in the fixed eigenbasis as
$h_t=P(a_t,b_t)^\top$. Its update $h_{t+1}=J_th_t$ has coefficient form
\[
  \begin{pmatrix}a_{t+1}\\b_{t+1}\end{pmatrix}
  =P^{-1}J_tP\begin{pmatrix}a_t\\b_t\end{pmatrix}.
\]
Consider the cone $\mathcal C\doteq\{(a,b)\in\R^2:|a|\le|b|\}$
in coefficient space. For $(a_t,b_t)\in\mathcal C$, the entrywise bounds give
\[
  |a_{t+1}|\le(1+2\epsilon_t)|b_t|,
  \qquad
  |b_{t+1}|\ge(2-2\epsilon_t)|b_t|.
\]
Since $\epsilon_t\le1/4$, the first bound is no larger than the second, so
$\mathcal C$ is invariant. Starting from $(a_0,b_0)=(0,1)$ gives
\begin{equation}\label{eq:app-cone-lower}
  |b_T|
  \ge2^T\prod_{t<T}(1-\epsilon_t).
\end{equation}
For $0\le u\le1/2$, $\log(1-u)\ge-2u$. Since
$\sum_{t<T}\epsilon_t=O(\log(T+1))$, the product is bounded below by a
negative power of $T+1$. Changing coordinates back through $P$ yields
\[
  \|J_{T-1}\cdots J_0\|_2
  \ge a2^T(T+1)^{-C}.
\]
For the upper bound, the induced maximum norm satisfies
\[
  \|P^{-1}J_tP\|_\infty
  \le2(1+\epsilon_t).
\]
Multiplication and $\log(1+u)\le u$ yield
\[
  \|J_{T-1}\cdots J_0\|_2
  \le A2^T(T+1)^C.
\]
Equivalence of norms changes only the constants $a,A,C$, proving
\eqref{eq:halbey-one-bit-law}.
\end{proof}

For each fixed $T$, differentiability gives the uncertainty interpretation
used in \eqref{eq:halbey-float-bits}:
\[
 \lim_{u\downarrow0}\frac1u
 \sup_{\|h\|_2\le u}\max_{0\le t\le T}
 \|F^t((r_0,\theta_0)+h)-F^t(r_0,\theta_0)\|_2
 =\max_{0\le t\le T}\|D(F^t)(r_0,\theta_0)\|_2.
\]
The right side is $2^TT^{O(1)}$ by the proposition. Its binary logarithm gives
the linearized precision requirement. 

The same leading exponent holds in Cartesian coordinates. Rotate to
$p=(1,0)$ and fix one side of the invariant plane. The coordinate map is
\[
  (r,\theta)\longmapsto
  x\doteq(1+r\theta,\,r\sqrt{1-\theta^2}).
\]
Its Jacobian columns are orthogonal, with singular values
$1$ and $r/\sqrt{1-\theta^2}$. Along the slow curve the latter is
$\Theta(\sqrt r)=\Theta(t^{-1/2})$. Thus coordinate conversion contributes
only polynomial factors, changing the required number of binary digits by
$O(\log T)$.

\subsubsection{Nonlinear bounds for additive state perturbations}

We next consider a nonlinear additive perturbation error model of the two polar
state coordinates. 

\begin{theorem}[Finite-horizon additive perturbation bound]
\label{thm:halbey-nonlinear-tube}
Let $z_t\doteq(r_t,\theta_t)$ be a slow segment from
\cref{prop:halbey-bit-law}, reindexed after entering a sufficiently
small fixed neighborhood of $(0,-1)$. There are constants
$c,\eta_0,\Gamma>0$, independent of $T$, with the following property. For
every integer $T\ge1$ and $0<\eta\le\eta_0$, let
\[
  \widetilde z_0\doteq z_0+e_0,
  \qquad
  \widetilde z_{t+1}\doteq F(\widetilde z_t)+\xi_t,
  \quad 0\le t<T,
\]
where $e_0,\xi_t\in\mathbb R^2$ may be chosen adversarially. If
\begin{equation}\label{eq:nonlinear-total-budget}
  \|e_0\|_\infty+
  \sum_{t=0}^{T-1}\|\xi_t\|_\infty
  \le c\eta\,2^{-T}(T+1)^{-\Gamma},
\end{equation}
then every perturbed state is in the domain of the polar map and
\begin{equation}\label{eq:nonlinear-shadow}
  \|\widetilde z_t-z_t\|_\infty
  \le \eta r_t^3
  =O\!\left(\frac{\eta}{(t+1)^3}\right),
  \qquad 0\le t\le T.
\end{equation}
In particular, \eqref{eq:nonlinear-total-budget} follows from
\[
  \|e_0\|_\infty
  \le \frac{c\eta}{2}\,2^{-T}(T+1)^{-\Gamma},
  \qquad
  \max_{t<T}\|\xi_t\|_\infty
  \le \frac{c\eta}{2}\,2^{-T}(T+1)^{-(\Gamma+1)}.
\]
\end{theorem}

\begin{proof}
Put $q\doteq1+\theta$. This translates the second state coordinate, leaving
differences and derivative matrices unchanged. Along the slow segment,
$q_t=r_t/3+O(r_t^2)$ and $r_t=\Theta((t+1)^{-1})$. After shrinking the fixed
entry neighborhood, choose constants $0<a<b$ such that
\[
  ar_t\le q_t\le br_t
\]
at every reindexed state. The numerator and denominator of the contraction
factor are
\begin{equation}\label{eq:wedge-numerator-denominator}
  1-\theta^2=2q-q^2,
  \qquad
  D(r,\theta)=r^2+2(1+r)q.
\end{equation}
Use a larger fixed radial cutoff $r_{\max}\doteq2r_+$, with $r_+$ sufficiently
small, and consider
$0<r\le r_{\max}$, $ar/2\le q\le2br$. This cutoff is distinct from the
actual initial radius $r_0$. Both quantities in
\eqref{eq:wedge-numerator-denominator} are bounded above and below by positive
constant multiples of $r$ throughout this larger neighborhood.

In the $(r,q)$ coordinates, with the output's second coordinate translated
as well, the same map is represented by
\[
  F(r,q)=\left(
    r\sqrt{\frac{2q-q^2}{r^2+2(1+r)q}},
    1-(1+r)\sqrt{\frac{2q-q^2}{r^2+2(1+r)q}}
  \right).
\]
Differentiating twice gives
\begin{equation}\label{eq:nonlinear-hessian-bound}
  \|D^2F(r,q)\|_\infty\le K r^{-2}
\end{equation}
for a fixed $K$. The first two derivatives of the polynomial numerator and
denominator are bounded; differentiating their square-root quotient twice
introduces at worst two reciprocal powers of quantities comparable to $r$.

Write $h_t\doteq\widetilde z_t-z_t$ and provisionally assume
$\|h_t\|_\infty\le\eta r_t^3$. For sufficiently small $\eta_0$, the segment
from $z_t$ to $z_t+h_t$ stays in the larger neighborhood, with radius
comparable to $r_t$. To see the margins, $r_t\le r_+$ leaves room up to
$r_{\max}=2r_+$ even for outward perturbations at time zero, while
$\|h_t\|_\infty/r_t\le\eta_0r_+^2$ is small relative to the fixed differences
between $ar_t\le q_t\le br_t$ and the enlarged angular bounds. The
integral mean-value formula therefore gives
\[
  F(z_t+h_t)-F(z_t)=A_t h_t,
  \qquad
  A_t=DF(z_t)+R_t,
  \qquad
  \|R_t\|_\infty\le K' r_t^{-2}\|h_t\|_\infty
  \le K'\eta r_t.
\]
In the fixed basis $P$ used above,
\[
  P^{-1}A_tP=\operatorname{diag}(1,-2)+\widetilde E_t,
  \qquad
  \|\widetilde E_t\|_\infty\le\frac{K''}{t+1}.
\]
The induced maximum norm and
$\prod_{s=j}^{t-1}(1+K''/(s+1))\le K_0(t+1)^{K''}$ yield fixed
$K_1,K_2>0$ such that every product formed while the provisional bound holds
satisfies
\begin{equation}\label{eq:nonlinear-product-upper}
  \|A_{t-1}\cdots A_j\|_\infty
  \le K_1 2^{t-j}(t+1)^{K_2}.
\end{equation}

Variation of constants in $h_{t+1}=A_th_t+\xi_t$ gives
\[
  \|h_t\|_\infty
  \le K_3 2^t(t+1)^{K_2}
  \left(\|e_0\|_\infty+
        \sum_{j=0}^{t-1}\|\xi_j\|_\infty\right).
\]
Take $\Gamma\doteq K_2+3$. The bound $r_t\ge c_r/(t+1)$ for a fixed $c_r>0$ and
\eqref{eq:nonlinear-total-budget} imply
\[
  \frac{\|h_t\|_\infty}{\eta r_t^3}
  \le \frac{K_3c}{c_r^3}
      2^{t-T}
      \left(\frac{t+1}{T+1}\right)^{K_2+3}
  \le \frac{K_3c}{c_r^3}.
\]
Choose $c<c_r^3/K_3$. A first-exit argument proves the provisional bound for
all $0\le t\le T$ and keeps every perturbed state in the larger neighborhood.
This gives \eqref{eq:nonlinear-shadow}. Finally, the separate bounds imply
\eqref{eq:nonlinear-total-budget} because $T/(T+1)\le1$.
\end{proof}

The fixed lower bound on the starting distance to the optimizer $r_-$ from
\cref{prop:halbey-bit-law} remains part of the theorem. Its
constants need not be uniform when the starting radius tends to zero
independently of $T$; already at $t=0$ the allowed error is $\eta r_0^3$.
The theorem only provides a sufficient bound for the displayed additive-error model.
\subsubsection{Zero area for the eventual slow envelope}
\label{app:halbey-measure-zero}

We prove \cref{thm:halbey-measure-zero}. Rotate so $p=(1,0)$ and put
$q\doteq1+\theta$. The coordinate maps $\chi_\pm$ recover the original
Cartesian point $x$ from $(r,q)$, with $+$ and $-$ selecting the upper and
lower half-disks, respectively:
\begin{equation}\label{eq:halbey-measure-chart}
  \chi_\pm(r,q)\doteq\bigl(1-r+rq,\ \pm r\sqrt{2q-q^2}\bigr).
\end{equation}
The interior of the disk, excluding the horizontal axis, corresponds to
\[
  \mathcal D\doteq\{(r,q):0<r<2,\ 0<q<1-r/2\}.
\]
Translating the second coordinate of \eqref{eq:halbey-polar-map} gives
\begin{equation}\label{eq:halbey-measure-map}
  \Phi(r,q)\doteq(rs,1-(1+r)s),\qquad
  s^2=\frac{N}{D},\quad
  N\doteq2q-q^2,\quad D\doteq r^2+2(1+r)q.
\end{equation}
The exact line-search parameter lies in $(0,1)$ throughout $\mathcal D$:
its numerator is $r(r+q)>0$, and its denominator exceeds that numerator
by $(r+2)q>0$. Thus clipping does not change this map. Moreover,
\begin{equation}\label{eq:halbey-measure-radial}
  D-(1+r)^2N=((1+r)q-r)^2\ge0.
\end{equation}
Hence $s\le1/(1+r)$ and $r_t\le(t+1/r_0)^{-1}$ until termination;
in particular, $r_t\to0$.

Put
\[
  \delta\doteq\frac1{100},\qquad a\doteq\frac14,\qquad b\doteq\frac25,
  \qquad
  W\doteq\{(r,q):0<r\le\delta,\ ar\le q\le br\}\subset\mathcal D.
\]
Let $W_0\doteq W$ and recursively define
\[
  W_{n+1}\doteq\{z\in W:\Phi(z)\in W_n\},\qquad
  W_\infty\doteq\bigcap_{n\ge0}W_n.
\]
Thus $W_n$ requires the first $n+1$ states, including the initialization,
to lie in $W$. Write $\mathcal L^2_{r,q}$ for ordinary area in the $(r,q)$
parameter plane, and $\mathcal L^2$ for ordinary area in the original
$(x_1,x_2)$ plane. The maps $\chi_\pm$ do not preserve area; their absolute
Jacobian determinants account for the change in area.

\begin{lemma}[Area of initializations remaining in $W$]
\label{lem:halbey-measure-area}
The map $\Phi$ is injective on $W$, and
\[
  \mathcal L^2_{r,q}(W_n)
  \le\frac{3\delta^2}{40}\left(\frac56\right)^n.
\]
In particular $\mathcal L^2_{r,q}(W_\infty)=0$.
\end{lemma}
\begin{proof}
Direct differentiation gives
\begin{equation}\label{eq:halbey-measure-jacobian}
  \det D\Phi(r,q)
  =-s\,\partial_qs
  =\frac{((1+r)q-r)(r+q)}{(r^2+2(1+r)q)^2}.
\end{equation}
With $k\doteq q/r\in[a,b]$, its magnitude on $W$ is
\[
  J(r,k)\doteq\frac{(1-(1+r)k)(1+k)}{(r+2(1+r)k)^2}.
\]
The numerator is positive. For fixed $k$, it decreases with $r$, while the
denominator increases. Also,
\[
  \partial_k J(r,k)=-\frac{(r+2)^2}{(r+2(1+r)k)^3}<0.
\]
Therefore
\begin{equation}\label{eq:halbey-measure-expansion}
  |\det D\Phi|\ge J(\delta,b)
  =\frac{208600}{167281}>\frac65.
\end{equation}

If $(R,Q)=\Phi(r,q)$, then $s=1-Q-R>0$ and $r=R/s$.
Once $r$ is recovered, $s^2$ is strictly increasing in $q$ on
$0<q<r/(1+r)$, since $\partial_q(s^2)=-2\det D\Phi>0$ there.
The whole of $W$ lies in that branch, so $q$ is unique. The same argument
gives a $C^1$ diffeomorphism on the open branch; the change-of-variables
formula therefore applies to its measurable subsets. Since
$\Phi(W_{n+1})\subseteq W_n$, \eqref{eq:halbey-measure-expansion} implies
\[
  \frac65\mathcal L^2_{r,q}(W_{n+1})
  \le\int_{W_{n+1}}|\det D\Phi|
  =\mathcal L^2_{r,q}(\Phi(W_{n+1}))
  \le\mathcal L^2_{r,q}(W_n).
\]
Finally $\mathcal L^2_{r,q}(W)=(b-a)\delta^2/2=3\delta^2/40$.
Iteration proves the bound, whose limit proves nullity.
\end{proof}

The corresponding Cartesian area satisfies
\begin{equation}\label{eq:halbey-measure-finite}
  \mathcal L^2\bigl(\chi_+(W_n)\cup\chi_-(W_n)\bigr)
  \le\frac{3\delta^{5/2}}{10}\left(\frac56\right)^n.
\end{equation}
Indeed $\left| \det D \chi_\pm \right|=r/\sqrt{2q-q^2}\le2\sqrt\delta$ on $W$;
apply \cref{lem:halbey-measure-area} to both maps. This bounds the area of
initial points $x_0$ whose reduced states satisfy $(r_t,q_t)\in W$ for every
$t=0,\ldots,n$, including the initial state.

\begin{lemma}[The eventual envelope and its asymptotic error]
\label{lem:halbey-measure-rate}
For a nonterminating trajectory, \eqref{eq:halbey-slow-envelope} holds
eventually if and only if its $(r_t,q_t)$ coordinates eventually remain in
$W$. On any such trajectory,
\[
  \frac{q_t}{r_t}=\frac13-\frac29r_t+o(r_t),\qquad
  s_t=1-\frac43r_t+(2+o(1))r_t^2,
  \qquad t r_t\longrightarrow\frac34.
\]
\end{lemma}
\begin{proof}
Suppose first that \eqref{eq:halbey-slow-envelope} holds. Once
$r_t\le\delta$, \eqref{eq:halbey-measure-map} gives
\[
  \frac{q_{t+1}}{r_{t+1}}
  =\frac{1/3+(4/3-c_t)r_t-c_tr_t^2}{s_t}.
\]
Here $74/75\le s_t\le1$, and interval bounds using $1\le c_t\le5/2$ give
\[
  \frac{3857}{12000}
  \le\frac{q_{t+1}}{r_{t+1}}
  \le\frac{101}{296}.
\]
Both endpoints lie strictly between $a$ and $b$ defined further above. Together with
\eqref{eq:halbey-measure-radial}, this proves eventual membership in $W$.

Conversely, suppose the trajectory eventually remains in $W$. Discard the
finite initial segment before this permanent stay and number the remaining
states starting at $t=0$. Set $k_t\doteq q_t/r_t$. The exact recurrence implies
\begin{equation}\label{eq:halbey-measure-increment}
  \frac1{r_{t+1}}-\frac1{r_t}=1+k_{t+1},\qquad
  k_{t+1}=G(r_t,k_t),
\end{equation}
where the following expression extends smoothly to $r=0$ for $k\in[a,b]$:
\[
  G(r,k)\doteq\frac{(1+k)^2}{\sqrt{A}(\sqrt A+\sqrt B)}-1,
  \qquad A\doteq k(2-kr),\quad B\doteq2k+(1+2k)r.
\]
Write $\widehat s(r,k)\doteq\sqrt{A/B}$ for the contraction in these coordinates,
so $s_t=\widehat s(r_t,k_t)$. This distinguishes it from the original
function $s(r,\theta)$ in \eqref{eq:halbey-polar-map}.
At $r=0$ the map $G$ is $g(k)\doteq(k-1)^2/(4k)$, so uniformly on this interval
$G(r,k)=g(k)+O(r)$. The fixed point $k_*\doteq1/3$ satisfies
\[
  g(k_*)=k_*,\qquad g'(k)\le-\frac{21}{16}\quad(a\le k\le b).
\]
Set $\lambda\doteq21/16$ and apply the mean-value theorem to $g$ between
$k_t$ and $k_*=1/3$. Since $g(k_*)=k_*$ and $|g'(k)|\ge\lambda$ on $[a,b]$,
\[
  \lambda|k_t-k_*|\le|g(k_t)-g(k_*)|
  \le|k_{t+1}-k_*|+C r_t.
\]
The second inequality uses $k_{t+1}=G(r_t,k_t)$ and the uniform estimate
$|G(r,k)-g(k)|\le Cr$.
Iterate this inequality forward. Since $k_t$ is bounded and
$r_{t+j}\le r_t$, letting the number of iterations tend to infinity yields
$|k_t-k_*|\le C r_t/(\lambda-1)$. Thus
$\beta_t\doteq(k_t-1/3)/r_t$ is bounded.

Uniform Taylor expansion for bounded $\beta$ gives
\[
  G(r,1/3+\beta r)
  =\frac13+\left(-2\beta-\frac23\right)r+O(r^2),
\]
\[
  \widehat s(r,1/3+\beta r)
  =1-\frac43r+\left(2\beta+\frac{22}{9}\right)r^2+O(r^3).
\]
As $r_{t+1}/r_t=1+O(r_t)$, the first expansion implies
\[
  \beta_{t+1}=-2\beta_t-\frac23+O(r_t).
\]
Let $L\doteq\limsup_t|\beta_t+2/9|<\infty$. Since $r_t\to0$, this identity gives
$2L\le L$, whence $L=0$. Thus $\beta_t\to-2/9$, and the second expansion
gives $c_t\to2$, proving eventual \eqref{eq:halbey-slow-envelope}.
Finally averaging \eqref{eq:halbey-measure-increment} gives
$(1/r_t)/t\to1+1/3=4/3$.
\end{proof}

\begin{lemma}[Existence at every small positive radius]
\label{lem:halbey-measure-existence}
For every $r_0\in(0,\delta]$ there is a $q_0\in[ar_0,br_0]$ whose
trajectory remains in $W$ for all iterations.
\end{lemma}
\begin{proof}
The angular map crosses both ends of the interval:
\begin{equation}\label{eq:halbey-measure-crossing}
  G(r,a)>b,\qquad G(r,b)<a\qquad(0<r\le\delta).
\end{equation}
For example $G(r,a)>b$ is equivalent to
$B-A(1+(1+b)r)^2>0$ at $k=a$. After dividing by $r$, that difference is
$(65-322r+49r^2)/400>0$ on the specified interval. Similarly
$G(r,b)<a$ follows from
\[
  \frac{A(1+(1+a)r)^2-B}{r}\bigg|_{k=b}
  =\frac{4+85r-25r^2}{100}>0.
\]
All factors squared in these comparisons are positive.

Start with the closed interval $J_0\doteq[a,b]$ of possible $k_0$ at the fixed
$r_0$. Inductively choose a nonempty closed subinterval $J_n$ such that all
states through time $n$ lie in $W$, and $k_n$ takes the values $a,b$ at its
two endpoints, in either order. The update is continuous on $J_n$.
By \eqref{eq:halbey-measure-crossing}, $k_{n+1}$ takes values on opposite
sides of $[a,b]$ at those endpoints. The intermediate-value theorem gives a
closed subinterval $J_{n+1}\subseteq J_n$ on which $k_{n+1}\in[a,b]$, with
endpoint values $a,b$. One may take a first crossing of one endpoint value
and the last crossing of the other value before it. All previous constraints
remain satisfied, and \eqref{eq:halbey-measure-radial} preserves the radial
cutoff. Compactness gives a point in $\bigcap_nJ_n$.
\end{proof}

\begin{proof}[Proof of \cref{thm:halbey-measure-zero}]
On $\mathcal D$, the Jacobian \eqref{eq:halbey-measure-jacobian} vanishes
only on the curve $q=r/(1+r)$, a set of area zero. Consequently the preimage
of every area-zero set is area zero: away from this curve the inverse-function
theorem provides local $C^1$ inverse maps. Countably many such neighborhoods
suffice, and their locally Lipschitz inverse maps preserve null sets. Add
the null critical curve to cover the remaining points.

By \cref{lem:halbey-measure-area}, every finite preimage of
$W_\infty$ in $\mathcal D$, and their countable union, has zero $(r,q)$ area.
\Cref{lem:halbey-measure-rate} identifies this union with the eventually
slow trajectories starting off the horizontal axis and the disk boundary.
A trajectory hitting the critical curve reaches the horizontal axis on the
next step, and then the optimizer, so it is not an omitted eventually slow
trajectory. The Cartesian charts \eqref{eq:halbey-measure-chart} are locally
$C^1$ on $\mathcal D$ and preserve null sets. The excluded initial axis and
disk boundary are already Cartesian null sets. Thus
$\mathcal L^2(\mathcal E)=0$.

\Cref{lem:halbey-measure-existence} supplies a point of $\mathcal E$ at
every initial radius $r_0\in(0,\delta]$, by
\cref{lem:halbey-measure-rate}. Distinct radii give distinct
initializations, proving uncountability. The rate lemma also gives
$t^2f(x_t)=(t r_t)^2\to9/16$. Absolute continuity proves the probability
claim.
\end{proof}

\section{Geometric parameters and the zero-call benchmark}
\label{app:parameter-comparison}

To compare lower and upper bounds on the same family, let $D=\diam(S)$ denote the rescaled diameter. The normalization in \cref{thm:main-paper} gives
\begin{equation*}
 D=\frac{2(2-\sqrt2)}{\alpha\sqrt{d^2+d+2}},\qquad
 \mu=L,\qquad d=64(T+1).
\end{equation*}
Thus the threshold $cLD^2/(T+1)^2$ has order
$L/(\alpha^2(T+1)^4)$ for fixed $L,\alpha$; keeping $D$ fixed instead forces
$\alpha=\Theta(1/(T+1))$. This coupling is inherited from
\citet{grimmer2026lower}, as we reuse their family rather than an artifact introduced by our randomized extension.

\subsection{Comparison with Garber--Hazan}

For the Frank--Wolfe iterates of \citet[Theorem~2]{garber2015faster},
the corresponding upper guarantee is
\begin{equation*}
 f(x_T)-f^\star\le
 \frac{\max\{(9/2)LD^2,\,2304L^2/(\alpha^2\mu)\}}{(T+2)^2}.
\end{equation*}
With $\mu=L$, its curvature-dependent term has order
$L/(\alpha^2(T+1)^2)$ and dominates the diameter-dependent term on this
family. As in \citet{grimmer2026lower}, our lower bound matches the order
of the latter term, not the full upper guarantee. This now naturally leads to the question whether the upper or the lower bound is not tight. We will now show that in fact the upper bound is loose for the family of instances considered here.

\subsection{A benchmark requiring no oracle calls}

Since $0\in S_\pi$ for every permutation, a feasible solution requiring no oracle calls is available. For the unscaled quadratic, \eqref{eq:objective} and
\eqref{eq:optimum} give its gap $g_0=f(0)-f^\star$ as
\begin{equation*}
 g_0=M\rho\sum_{k=1}^d(w_d-w_k)
       -\frac{\rho^2}{2}\sum_{k=1}^d(w_d-w_k)^2.
\end{equation*}
Dividing the boundary equation \eqref{eq:rho} by $C^2$ and using
$\sum_{k=1}^d\sqrt{k}\sim(2/3)d^{3/2}$ yields $\rho\sim3/d^2$ and hence
$g_0\sim9/(2d^4)$.
Rescaling multiplies this gap by $L/\alpha^2$. Since
$D^2\sim4(2-\sqrt2)^2/(\alpha^2d^2)$ and $T+1=d/64$,
both the rescaled gap and our lower bound are asymptotic to positive
constants times $L/(\alpha^2d^4)$. This common scale cancels in their
ratio, leaving only the ratio of the leading coefficients. Thus, as
$T\to\infty$,
\begin{equation*}
 \frac{(L/\alpha^2)g_0}{cLD^2/(T+1)^2}
 \longrightarrow\frac{9}{8c\,64^2(2-\sqrt2)^2}=3456.
\end{equation*}
For short-step or exact-line-search Frank--Wolfe initialized at $0$,
monotonicity bounds the error on the rescaled instance by
$(L/\alpha^2)g_0$ at every iteration. Retaining this initial error
therefore gives an upper bound of the same order as our lower
bound on this family.

The theorem therefore rules out more than a constant-factor improvement
in worst-case expected error over the benchmark arising from the feasible
solution $0$ (which does not require any oracle calls) if no more than $T$
oracle calls are permitted.

\section{Example of compatible assignments and swaps}
\label{app:assignment-example}

\begin{example}
\label{ex:assignments}
Take $d\doteq5$ and the generating permutation $(2,4,1,3,5)$.
Use nonnegative query directions with squared components
$(5/2,9/2,0,0,0)$ and then $(3,0,0,2,0)$. Both observed thresholds are $1$.
The first answer reveals $\Pi(1)=2,\Pi(2)=4$ simultaneously; the second
adds $\Pi(4)\ge2$ while its only active coordinate is already known.
The remaining ranks are $(r_1,r_2,r_3)=(1,3,5)$, and sorting the compressed
bounds gives $\iota=(3,5,4)$ and $\ell=(1,1,2)$.
There are four compatible reduced permutations:
$(1,2,3),(1,3,2),(2,1,3),(3,1,2)$.
\begin{figure}[t]
 \centering
 \includegraphics[width=.82\linewidth]{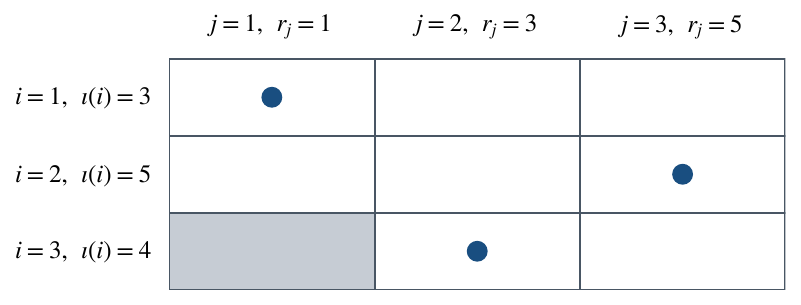}
 \caption{The example's allowed assignments. Columns show both reduced
 indices $j$ and original ranks $r_j$; rows show both $i$ and $\iota(i)$.
 The shaded cell is excluded by the second answer. Dots mark the generating
 assignment $\sigma=(1,3,2)$ (in the reduced coordinates), not the only compatible assignment.}
 \label{fig:assignment-board}
\end{figure}
\end{example}

\subsection{Swap graph and one-query update}
\label{app:swap-example}

We continue \cref{ex:assignments}, with remaining original ranks
$(1,3,5)$, physical row order $\iota=(3,5,4)$, and bounds $\ell=(1,1,2)$.
The four reduced permutations are the vertices of
\cref{fig:assignment-graph}. From $(1,2,3)$, exchanging the assignments in
rows $1,2$ swaps original ranks $1,3$ between physical coordinates $3,5$,
giving $(2,1,3)$. The swap in rows $2,3$ gives $(1,3,2)$.
The swap in rows $1,3$ would give $(3,2,1)$, violating $\ell_3=2$.
Every vertex has $\Delta=(1+2+3)-(1+1+2)=2$ outgoing edges,
hence eight directed edges in total.

Now query $p\doteq e_3$, the unit vector at physical coordinate $3$,
which is reduced row $1$. Write $k\doteq\Pi(3)=r_{\sigma(1)}\in\{1,3,5\}$.
Only this coordinate can be active, so the threshold equation is
$(\tau-k)_+=1$. Since $C=1/\sqrt{32}$, the threshold and LMO answer are
\[
 \tau_\sigma=k+1,\qquad
 z(k)\doteq-\frac{\sqrt{k+1}-\sqrt{k}}4\,e_3.
\]
These three numerical answers are distinct. Hence two vertices give the
same answer exactly when they have the same first reduced column $\sigma(1)$.

After the answer, remove coordinate $3$ and its disclosed rank $k$.
The other query components are zero, so the old bounds $b_4=2$ and $b_5=1$
are unchanged. Recompute the lower bounds on list indices and sort the rows
as before. In the table, the answer classes use the old reduced indices;
the last four columns describe the updated state.
\[
\begin{array}{c|c|c|c|c|c}
 k & \text{answer class} &
 \substack{\text{remaining}\\\text{original ranks}} &
 \substack{\text{physical}\\\text{row order}} & \ell^+ & \Delta^+\\
 \hline
 1 & \{(1,2,3),(1,3,2)\} & (3,5) & (4,5) & (1,1) & 1\\
 3 & \{(2,1,3)\} & (1,5) & (5,4) & (1,2) & 0\\
 5 & \{(3,1,2)\} & (1,3) & (5,4) & (1,2) & 0
\end{array}
\]
For $k=1$, both bounds select the first remaining rank $3$, so both new
index bounds are $1$. The physical-coordinate tie rule puts coordinate $4$
before $5$: the new row order is not obtained merely by deleting the first
entry of $(3,5,4)$. The updated family is
$\mathcal F((1,1))=\{(1,2),(2,1)\}$.
For $k=3$ or $5$, coordinate $4$ must receive the larger remaining rank.
Thus $\mathcal F((1,2))=\{(1,2)\}$: two coordinates remain inactive,
but their ranks are determined by the constraints.

Under the conditional uniform law on the four vertices, the outcomes
$k=1,3,5$ have probabilities $1/2,1/4,1/4$. Consequently
\[
 \E\Delta^+=\tfrac12\cdot1+\tfrac14\cdot0+\tfrac14\cdot0=\tfrac12,
 \qquad
 \E[\Delta-\Delta^+]=2-\tfrac12=\tfrac32.
\]
Only the two directed edges between $(1,2,3)$ and $(1,3,2)$ stay within
an answer class. The other six cross classes; averaging over the four
vertices gives the same decrease $6/4=3/2$.
The generating permutation has $k=1$, so its realized decrease is $1$,
rather than the conditional average $3/2$.
\begin{figure}[t]
 \centering
 \includegraphics[width=.72\linewidth]{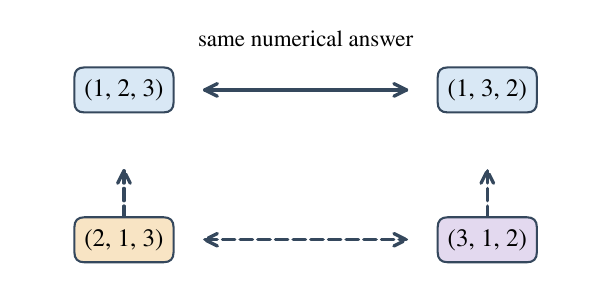}
 \caption{Vertices are the four compatible reduced permutations.
 Each double-headed connection represents two directed swap edges.
 Equal colors indicate equal numerical answers to $p=e_3$.
 Dashed connections cross answer classes; only the top pair remains together.}
 \label{fig:assignment-graph}
\end{figure}

 \end{document}